\documentclass[fleqn]{article}

\usepackage{amscd,amsmath,amssymb,amsthm}

\newcommand{\abs}[1]{\left|#1\right|}
\newcommand{\bdry}[1]{\partial #1}
\newcommand{\bgset}[1]{\big\{#1\big\}}
\newcommand{\closure}[1]{\overline{#1}}

\newcommand{\dint}{\ds{\int}}
\newcommand{\ds}[1]{\displaystyle #1}
\newcommand{\eps}{\varepsilon}

\newcommand{\incl}{\subset}
\newcommand{\norm}[2][]{\left\|#2\right\|_{#1}}
\newcommand{\pnorm}[2][]{\if #1'' \left|#2\right| \else \left|#2\right|_{#1} \fi}
\newcommand{\QED}{\mbox{\qedhere}}
\newcommand{\restr}[2]{\left.#1\right|_{#2}}
\newcommand{\set}[1]{\left\{#1\right\}}

\newcommand{\A}{{\cal A}}
\newcommand{\C}{{\cal C}}
\newcommand{\D}{{\cal D}}
\newcommand{\F}{{\cal F}}
\newcommand{\M}{{\cal M}}

\newcommand{\R}{\mathbb R}
\newcommand{\RP}{\R \text{P}}
\newcommand{\N}{\mathbb N}
\newcommand{\Z}{\mathbb Z}

\newcommand\meas{{\rm meas}}
\newcommand{\J}{{\cal J}}

\DeclareMathOperator{\dist}{dist}
\DeclareMathOperator{\divg}{div}
\DeclareMathOperator*{\esssup}{ess\; sup}

\newenvironment{properties}[1]{\begin{enumerate}
\renewcommand{\theenumi}{$(#1_\arabic{enumi})$}
\renewcommand{\labelenumi}{$(#1_\arabic{enumi})$}}{\end{enumerate}}

\newtheorem{corollary}{Corollary}[section]
\newtheorem{theorem}[corollary]{Theorem}
\newtheorem{lemma}[corollary]{Lemma}
\newtheorem{proposition}[corollary]{Proposition}

\theoremstyle{definition}

\newtheorem{remark}[corollary]{Remark}

\numberwithin{equation}{section}

\begin{document}

\title{\bf Multiple solutions\\
 for $p$-Laplacian type problems\\
with asymptotically $p$-linear terms\\
via a cohomological index theory}
\author{{\bf\large A.M. Candela\footnote{The authors acknowledge the support
of Research Funds PRIN2009 and {\sl Fondi d'Ateneo 2010}.} $^{,1}$, G. Palmieri$^{*,2}$,
K. Perera\footnote{This work was done while the third--named
author was visiting the {\sl Dipartimento di Matematica} at the
{\sl Universit\`a degli Studi di Bari},
and he is grateful for the kind hospitality of the department.}}\\
{\small $^*$Dipartimento di Matematica}\\
{\small Universit\`a degli Studi di Bari ``Aldo Moro''} \\
{\small Via E. Orabona 4, 70125 Bari, Italy}\\
{\small \it $^1$annamaria.candela@uniba.it, $^2$giuliana.palmieri@uniba.it}
\vspace{1mm}\\
{\small $^\dagger$Department of Mathematical Sciences}\\
{\small Florida Institute of Technology}\\
{\small 150 W. University Blvd, Melbourne, FL 32901, USA}\\
{\small \it kperera@fit.edu}
\vspace{1mm}}
\date{}

\maketitle

\begin{abstract}
The aim of this paper is investigating the existence of weak solutions of the quasilinear
elliptic model problem
\[
\left\{
\begin{array}{lr}
- \divg (A(x,u)\, |\nabla u|^{p-2}\, \nabla u) + \dfrac1p\, A_t(x,u)\, |\nabla u|^p\  =\  f(x,u) & \hbox{in $\Omega$,}\\
u\ = \ 0 & \hbox{on $\partial\Omega$,}
\end{array}
\right.
\]
where $\Omega \subset \R^N$ is a bounded domain, $N\ge 2$,
$p > 1$, $A$ is a given function which admits partial derivative $A_t(x,t) = \frac{\partial A}{\partial t}(x,t)$
and $f$ is asymptotically $p$-linear at infinity.

Under suitable hypotheses both at the origin and at infinity, and
if $A(x,\cdot)$ is even while $f(x,\cdot)$ is odd, by using
variational tools, a cohomological index theory and a related pseudo--index
argument, we prove a multiplicity result if $p > N$ in the non--resonant case.
\end{abstract}



\section{Introduction} \label{secintroduction}

Let us consider the $p$-Laplacian type equation
\[
(P)\qquad
\left\{
\begin{array}{lr}
- \divg (A(x,u)\, |\nabla u|^{p-2}\, \nabla u) + \dfrac1p\, A_t(x,u)\, |\nabla u|^p\  =\  f(x,u) & \hbox{in $\Omega$,}\\
u\ = \ 0 & \hbox{on $\partial\Omega$,}
\end{array}
\right.
\]
where $\Omega \subset \R^N$ is a bounded domain, $N\ge 2$,
$p > 1$, $A$, $f : \Omega \times \R \to \R$
are given functions such that the partial derivative $A_t(x,t) = \frac{\partial A}{\partial t}(x,t)$
exists for a.e. $x\in \Omega$, all $t \in \R$.

If we set $F(x,t) = \int_0^t f(x,s) ds$, we can associate with problem $(P)$
the functional $\J : \D \subset W^{1,\, p}_0(\Omega) \to \R$ defined by
\begin{equation} \label{gei0}
\J(u)\ =\ \frac1p\ \int_\Omega A(x,u)\ |\nabla u|^p\, dx\ -\  \int_\Omega F(x,u)\ dx.
\end{equation}

In general, if no growth assumption is made on $A$ with respect to $t$,
the natural domain $\D$ of $\J$ is contained in, but is not equal to, the
Sobolev space $W^{1,\, p}_0(\Omega)$. Anyway, under the assumptions
\begin{description}{}{}
\item[$(H_0)$]
$A$, $A_t$ are Carath\'eodory functions on $\Omega\times\R$ such that
\[
\sup_{|t| \le r} |A(\cdot,t)| \in L^\infty(\Omega),\qquad
\sup_{|t| \le r} |A_t(\cdot,t)| \in L^\infty(\Omega)\qquad \hbox{for any $r > 0$;}
\]
\item[$(h_0)$]
$f$ is a Carath\'eodory function on $\Omega\times\R$ such that
\[
\sup_{|t| \le r} |f(\cdot,t)| \in L^\infty(\Omega)\qquad\hbox{for any $r > 0$,}
\]
\end{description}
the functional $\J$ is surely well-defined on the Banach space
\begin{equation}\label{space}
X := W^{1,\, p}_0(\Omega) \cap L^\infty(\Omega),\qquad
\|u\|_X = \|u\| + |u|_\infty,
\end{equation}
with
\[
\|u\|^p\ =\ \int_\Omega |\nabla u|^p\, dx, \quad
|u|_{\infty}\ =\ \esssup_{x\in\Omega} |u(x)|,
\]
and, for any $u, v \in X$, its G\^ateaux derivative with respect to $u$ in the direction $v$ is given by
\[
\begin{split}
\langle d\J(u),v\rangle\ =\ &\int_\Omega A(x,u)\ |\nabla u|^{p-2}\ \nabla u \cdot \nabla v\ dx\\
&+\ \frac1p\ \int_\Omega A_t(x,u)\ |\nabla u|^p\ v\ dx\ -\ \int_\Omega f(x,u)\ v\ dx.
\end{split}
\]

As our aim is investigating the existence of weak solutions of $(P)$
when it is an {\sl asymptotically $p$-linear} elliptic problem, we assume that $A$ and $f$ satisfy
the following hypotheses:
\begin{description}{}{}
\item[$(H_1)$] there exists $\alpha_0 > 0$ such that
\[
A(x,t) \ge \alpha_0 \quad \text{a.e. in } \Omega, \text{ for all } t \in \R;
\]
\item[$(H_2)$] there exists $A^\infty \in L^\infty(\Omega)$ such that
\[
\lim_{|t| \to + \infty}\, A(x,t) = A^\infty(x) \quad \text{uniformly a.e. in } \Omega;
\]
\item[$(h_1)$] there exist $\lambda^\infty \in \R$ and a (Carath\'eodory) function
$g^\infty : \Omega \times \R \to \R$ such that
\[
f(x,t)\ =\ \lambda^\infty\ |t|^{p-2}\ t + g^\infty(x,t),
\]
where
\begin{equation}\label{suglim}
\lim_{|t| \to + \infty}\, \frac{g^\infty(x,t)}{|t|^{p-1}} = 0 \quad \text{uniformly a.e. in } \Omega.
\end{equation}
\end{description}

As $\J$ is a $C^1$-functional on $X$ under these hypotheses
(see Proposition \ref{smooth}), we can seek weak solutions of $(P)$ by means
of variational tools.

In the asymptotically linear
case, i.e. under the hypotheses $(h_0)$ and $(h_1)$, a variational approach was first used
for $p=2$ and $A(x,t)\equiv 1$ (see the seminal papers \cite{az,BBF}).
On the contrary, only a few results have been obtained when $p\ne 2$, but always for
$A(x,t)\equiv 1$ or, at worst, for $A(x,t) = A(x)$ independent of $t$
(see \cite{ag,ao,BCS,cm,dr,ll,lz,lz1,PAO,ps}). In fact,
when $p > 1$ is arbitrary, the main difficulty is that,
while the structure of the spectrum of $-\Delta$ in $H^{1}_0(\Omega)$
is known, the full spectrum of $-\Delta_p$ is still unknown,
even though various authors have introduced different
characterizations of eigenvalues and definitions of quasi--eigenvalues.

Clearly, the same problem arises in our setting when $A(x,t)$ depends on $t$.
Furthermore, we have difficulties with the Palais--Smale condition as well,
and have to consider the asymptotic behavior, both at the origin and at infinity,
not only of the term $f(x,t)$, but also of the coefficient $A(x,t)$.

When $(h_1)$ is replaced with different conditions at infinity, weaker versions
of the Palais--Smale condition hold for arbitrary $p > 1$, and the existence
of critical points of $\J$ in $X$ have been proved (see \cite{CP,CP4}). However,
these approaches do not distinguish between different critical points at the same
critical level (see \cite{CP2,CP3}), and therefore, up to now, multiplicity results via a
cohomological index theory have been obtained only for $p>N$ (see \cite{CMPP,CPP}).
In fact, in this case the Sobolev Imbedding Theorem implies $X = W^{1,p}_0(\Omega)$ and
the classical Cerami's variant of the Palais--Smale condition can be verified.

In this paper, we will prove a multiplicity result for problem $(P)$ when $p>N$
and $f(x,t)$ is asymptotically $p$-linear at infinity.
To this aim, by considering some sequences of eigenvalues defined
by means of the cohomological index, we will
prove the classical Palais--Smale condition and, by means of a cohomological
index theory and a related pseudo--index argument,
we will extend the result in \cite{ps} to our setting (see \cite{CP13}
for a result obtained by using the approach in \cite{BBF}).
In particular, let us point out that, if the coefficient $A$
depends on $t$, the boundedness of each Palais--Smale sequence of $\J$
requires a careful proof also in the non--resonant assumption,
unlike the $t$--independent case (see Proposition \ref{PSbound}).


\section{Abstract tools} \label{secabstract}

The aim of this section is to recall the abstract tools
we need for the proof of our main result. Hence,
let $({\cal B},\|\cdot\|_{\cal B})$ be a Banach space with dual space
$({\cal B}',\|\cdot\|_{\cal B'})$ and let $J \in C^1({\cal B},\R)$.

Furthermore, fixing a level $\beta \in \R$, a point $u_0 \in {\cal B}$, a
set $\C \subset {\cal B}$ and a radius $r > 0$,
let us denote
\begin{itemize}
\item $K^J = \{u \in {\cal B}:\ dJ(u) = 0\}$
the set of critical points of $J$ in ${\cal B}$;
\item $K^J_\beta = \{u \in {\cal B}:\ J(u) = \beta,\ dJ(u) = 0\}$
the set of critical points of $J$ in ${\cal B}$ at the level $\beta$ (clearly,
$K^J_\beta = \emptyset$ if $\beta$ is a regular value);
\item $J^\beta = \{u\in {\cal B}:\ J(u) \le \beta\}$
the sublevel set of $J$ associated with $\beta$;
\item $B^{\cal B}_r(u_0)\ =\ \{u \in {\cal B}:\ \|u - u_0\|_{\cal B}\ \le\ r\}$
the closed ball in ${\cal B}$ centered at $u_0$ of radius $r$, with boundary
$\partial B^{\cal B}_r(u_0)$;
\item $\displaystyle {\rm dist}_{\cal B}(u,{\C})\ =\ \inf_{v \in \C} \|v - u\|_{\cal B}$
the distance from $\C$ to $u \in {\cal B}$.
\end{itemize}

We say that a sequence $(u_n)_n\subset {\cal B}$ is a {\sl Palais--Smale sequence at the level $\beta$},
brieftly a {\sl $(PS)_\beta$--sequence}, if
\[
J(u_n) \to \beta  \quad \hbox{and}\quad \|dJ(u_n)\|_{{\cal B}'} \to 0\qquad \mbox{as $\ n\to+\infty$.}
\]
The functional $J$ satisfies the {\sl Palais--Smale condition at the level $\beta$} in ${\cal B}$,
{\sl $(PS)_\beta$ condition} for short, if every $(PS)_\beta$--sequence admits a subsequence that
converges in ${\cal B}$.

Now, we assume that $J$ is even and $J(0) = 0$, and use the $\Z_2$-cohomological
index of Fadell and Rabinowitz in \cite{FR} and the associated pseudo-index of
Benci in \cite{B} to obtain multiple critical points.

Let us first recall the definition and some basic properties of the cohomological index.

Let $\A$ be the class of symmetric subsets of ${\cal B} \setminus \set{0}$. For $A \in \A$, we denote by
\begin{itemize}
\item $\; \overline{A} = A/\Z_2$ the quotient space of $A$ with each $u$ and $-u$ identified,
\item $\; f : \overline{A} \to \RP^\infty$ the classifying map of $\overline{A}$,
\item $\; f^\ast : H^\ast(\RP^\infty) \to H^\ast(\overline{A})$ the induced homomorphism of
the Alexander--Spanier cohomology rings.
\end{itemize}
Then the cohomological index of $A$ is defined by
\[
i(A)\ =\ \begin{cases}
\sup \bgset{m \ge 1 : f^\ast(\omega^{m-1}) \ne 0} &\hbox{if}\, A \ne \emptyset,\\[3pt]
0 &\hbox{if}\, A = \emptyset,
\end{cases}
\]
where $\omega \in H^1(\RP^\infty)$ is the generator of the polynomial
ring $H^\ast(\RP^\infty) = \Z_2[\omega]$.

For example, if $S^{n-1}$ is the unit sphere in $\R^n$, $\ 1 \le n < +\infty$,
then $i(S^{n-1}) = n$ as the classifying map of $S^{n-1}$
is the inclusion $\RP^{n-1} \incl \RP^\infty$,
which induces isomorphisms on $H^q$ for $q \le n-1$.

\begin{proposition}[Fadell--Rabinowitz \cite{FR}]
The index $i : \A \to \N \cup \set{0,+\infty}$ has the following properties:
\begin{properties}{i}
\item \label{i1} Definiteness: $i(A) = 0$ if and only if $A = \emptyset$;
\item \label{i2} Monotonicity: If there is an odd continuous map from $A$ to $B$
(in particular, if $A \subset B$), then $i(A) \le i(B)$.
Thus, equality holds when the map is an odd homeomorphism;
\item \label{i3} Dimension: $i(A) \le \dim {\cal B}$;
\item \label{i4} Continuity: If $A \in \A$ is closed, then there is a closed neighborhood
$N \in \A$ of $A$ such that $i(N) = i(A)$. When $A$ is compact, $N$ may be chosen to be a
$\delta$-neighborhood $N_\delta(A) = \bgset{u \in {\cal B} : \dist_{\cal B}(u,A) \le \delta}$;
\item \label{i5} Subadditivity: If $A, B \in \A$ are closed, then $i(A \cup B) \le i(A) + i(B)$;
\item \label{i6} Stability: If $SA$ is the suspension of $A \ne \emptyset$, obtained as the
quotient space of $A \times [-1,1]$ with $A \times \set{1}$ and $A \times \set{-1}$ collapsed
to different points, then $i(SA) = i(A) + 1$;
\item \label{i7} Piercing property: If $A, A_0, A_1$ are closed
and $\varphi : A \times [0,1] \to A_0 \cup A_1$ is a continuous mapping
such that $\varphi(-u,t) = - \varphi(u,t)$ for all $(u,t) \in A \times [0,1]$,
$\varphi(A \times [0,1])$ is closed, $\varphi(A \times \set{0}) \subset A_0$,
and $\varphi(A \times \set{1}) \subset A_1$, then $i(\varphi(A \times [0,1]) \cap A_0 \cap A_1) \ge i(A)$;
\item \label{i8} Neighborhood of zero: If $U$ is a bounded closed symmetric neighborhood of $0$,
then $i(\bdry{U}) = \dim {\cal B}$.
\end{properties}
\end{proposition}

For any integer $k \ge 1$, let
\[
\A_k = \bgset{A \in \A : A \text{ is compact and } i(A) \ge k}
\]
and set
\[
c_k := \inf_{A \in \A_k}\, \max_{u \in A}\, J(u).
\]
Since $\A_{k+1} \subset \A_k$, then $c_k \le c_{k+1}$.
Furthermore, for any $k$-dimensional subspace $V$ of ${\cal B}$
and $\delta > 0$, by \ref{i8} we have $\bdry{B^{\cal B}_\delta(0)} \cap V \in \A_k$,
while by continuity it results
\[
\sup_{u \in \bdry{B^{\cal B}_\delta(0)}} J(u)\ \to\ J(0)\quad
\hbox{as $\; \delta \to 0\ $},
\]
so $c_k \le J(0)$.

The following theorem is standard (see, e.g., \cite[Proposition 3.36]{PAO}).

\begin{theorem} \label{Theorem 2.4}
Assume that $J \in C^1({\cal B},\R)$ is even and $J(0) = 0$. If
\[
-\infty\ <\ c_k\ \le\ \dotsb\ \le\ c_{k+m-1}\ <\ 0
\]
and $J$ satisfies $(PS)_{c_i}$ for $i = k,\dots,k+m-1$, then $J$ has $m$ distinct pairs of nontrivial critical points.
\end{theorem}

Now, let us recall the definition and some basic properties of a pseudo--index related to the cohomological index $i$.

Let $\A^\ast$ denote the class of symmetric subsets of ${\cal B}$, let $\M \in \A$ be closed,
and let $\Gamma$ denote the group of odd homeomorphisms $\gamma$ of ${\cal B}$ such that
$\restr{\gamma}{J^0}$ is the identity. Then the pseudo-index of $A \in \A^\ast$ related to
$i$, $\M$, and $\Gamma$ is defined by
\[
i^\ast(A) = \min_{\gamma \in \Gamma}\, i(\gamma(A) \cap \M).
\]

\begin{proposition}[Benci \cite{B}]
The pseudo-index $i^\ast : \A^\ast \to \N \cup \set{0,+\infty}$ has the following properties:
\begin{properties}{i^\ast}
\item \label{i*1} If $A \subset B$, then $i^\ast(A) \le i^\ast(B)$;
\item \label{i*2} If $\eta \in \Gamma$, then $i^\ast(\eta(A)) = i^\ast(A)$;
\item \label{i*3} If $A \in \A^\ast$ and $B \in \A$ are closed, then $i^\ast(A \cup B) \le i^\ast(A) + i(B)$.
\end{properties}
\end{proposition}

For any integer $k \ge 1$ such that $k \le i(\M)$, let
\[
\A_k^\ast\ =\ \bgset{A \in \A^\ast : A \text{ is compact and } i^\ast(A) \ge k}
\]
and set
\[
c_k^\ast\ :=\ \inf_{A \in \A_k^\ast}\, \max_{u \in A}\, J(u).
\]
From $\A_{k+1}^\ast \subset \A_k^\ast$, it follows $c_k^\ast \le c_{k+1}^\ast$.
The following theorem is standard (see, e.g., \cite[Proposition 3.42]{PAO}).

\begin{theorem} \label{Theorem 2.5}
Assume that $J \in C^1({\cal B},\R)$ is even and $J(0) = 0$. If
\[
0\ <\ c_k^\ast\ \le\ \dotsb\ \le\ c_{k+m-1}^\ast\ <\ +\infty
\]
and $J$ satisfies $(PS)_{c_i^\ast}$ for $i = k,\dots,k+m-1$, then $J$ has $m$ distinct pairs of nontrivial critical points.
\end{theorem}

\section{The Palais--Smale condition} \label{secwps}

From here on, let $X$ be the Banach space in \eqref{space} and let $\J: X \to \R$ be the
functional in \eqref{gei0}.
Furthermore, we denote by
\begin{itemize}
\item $(X',\|\cdot\|_{X'})$ the dual space of $(X,\|\cdot\|_{X})$,
\item $(W^{-1,p'}(\Omega),\|\cdot\|_{W^{-1}})$ the dual space of $(W^{1,p}_0(\Omega),\|\cdot\|)$,
\item $L^q(\Omega)$ the Lebesgue space equipped with the canonical norm $|\cdot|_q$ for any $q \ge 1$,
\item $\meas( \cdot)$ the usual Lebesgue measure in $\R^N$.
\end{itemize}
By definition, $X \hookrightarrow W^{1,p}_0(\Omega)$ and $X \hookrightarrow L^\infty(\Omega)$
with continuous imbeddings; moreover, if $p^*$ is the critical exponent, i.e.
$p^* = \frac{pN}{N-p}$ if $p\in [1,N[$, $p^* = +\infty$ otherwise,
by the Sobolev Imbedding Theorem, for any $1 \le q < p^*$, a constant
$\gamma_q > 0$ exists such that
\begin{equation}\label{sobq}
|u|_q\ \le\ \gamma_q \|u\|\quad \hbox{for all $u \in W^{1,p}_0(\Omega)$.}
\end{equation}
In particular, for $1 \le p < p^*$, we have
\begin{equation}\label{sobqbis}
|u|_p\ \le\ \gamma_p \|u\|,\quad |u|_1\ \le\ \gamma_1 \|u\| \qquad \hbox{for all $u \in W^{1,p}_0(\Omega)$,}
\end{equation}
while, under the stronger assumption $p > N$, we have
\begin{equation}\label{sobqter}
|u|_\infty\ \le\ \gamma_\infty\|u\| \qquad \hbox{for all $u \in W^{1,p}_0(\Omega)$.}
\end{equation}

Letting $g^\infty$ be as in $(h_1)$ and setting $G^\infty(x,t) = \int_0^t g^\infty(x,s) ds$,
if $(h_0)$ and $(h_1)$ hold, then $g^\infty$ is a Carath\'eodory function on $\Omega\times\R$ such that
\begin{equation}\label{sug}
\sup_{|t| \le r} |g^\infty(\cdot,t)| \in L^\infty(\Omega)\qquad\hbox{for any $r > 0$;}
\end{equation}
furthermore, \eqref{sug}, respectively \eqref{suglim}, implies that
\begin{equation}\label{sug1}
\sup_{|t| \le r} |G^\infty(\cdot,t)| \in L^\infty(\Omega)\qquad\hbox{for any $r > 0$,}
\end{equation}
\begin{equation}\label{suglim1}
\lim_{|t| \to + \infty}\, \frac{G^\infty(x,t)}{|t|^{p}} = 0 \quad \text{uniformly a.e. in } \Omega.
\end{equation}
Hence \eqref{suglim} and \eqref{sug}, respectively \eqref{sug1} and \eqref{suglim1},
imply that for any $\eps > 0$ a constant $L_\eps > 0$ exists such that
\begin{eqnarray}
\label{sug3}
&&|g^\infty(x,t)|\ \le\  L_\eps + \eps |t|^{p-1}\qquad\hbox{for a.e. $x \in \Omega$, all $t \in \R$,}\\[5pt]
\label{sug2}
&&|G^\infty(x,t)|\ \le\  L_\eps + \eps |t|^p\qquad\hbox{for a.e. $x \in \Omega$, all $t \in \R$.}
\end{eqnarray}

Throughout this section, we consider the parametrized family of functionals
$\J_\lambda : X \to \R$ defined by
\begin{equation}\label{withlambda}
\J_\lambda(u)\ =\ \frac1p \int_\Omega (A(x,u) |\nabla u|^p - \lambda |u|^p)\ dx -  \int_\Omega G^\infty(x,u)\ dx.
\end{equation}

\begin{proposition}\label{smooth}
Let $p \ge 1$ and assume that the conditions $(H_0)$, $(h_0)$ and $(h_1)$ hold.
If $(u_n)_n \subset X$, $u \in X$ are such that
\begin{equation}\label{succ1}
\|u_n - u\| \to 0 \quad\hbox{as $n\to+\infty$}
\end{equation}
and $k > 0$ exists so that
\begin{equation}\label{succ2}
|u_n|_\infty \le k\qquad \hbox{for all $n \in \N$,}
\end{equation}
then for any $\lambda \in \R$, we have
\[
\J_\lambda(u_n) \to \J_\lambda (u)\qquad \hbox{and}\qquad \|d\J_\lambda(u_n) - d\J_\lambda(u)\|_{X'} \to 0
\quad\hbox{as $\ n\to+\infty$.}
\]
In particular, $\J_\lambda \in C^1(X,\R)$ with derivative $d\J_\lambda: u \in X \mapsto d\J_\lambda(u) \in X'$ defined by
\[
\begin{split}
\langle d\J_\lambda(u),\varphi\rangle\ =\ & \int_\Omega A(x,u) |\nabla u|^{p-2} \nabla u \cdot \nabla \varphi\ dx
+ \frac1p \int_\Omega A_t(x,u) \varphi |\nabla u|^p dx\\
& - \lambda \int_\Omega |u|^{p-2}u\varphi\ dx
 - \int_\Omega g^\infty(x,u) \varphi\ dx, \end{split}
\]
for any $u$, $\varphi \in X$.
\end{proposition}

\begin{proof}
The proof is essentially a simpler version of \cite[Proposition 3.1]{CP},
but, for completeness, here we point out its main tools.\\
First of all, consider the functional $\bar\J : X \to \R$ which is defined as
\[
\bar\J(w)\ =\ \frac1p\ \int_\Omega A(x,w) |\nabla w|^p dx, \qquad w\in X,
\]
whose G\^ateaux differential in $w$ along direction $\varphi$ ($w$, $\varphi\in X$) is
\begin{eqnarray*}
\langle d\bar\J(w),\varphi\rangle &=& \int_\Omega A(x,w) |\nabla w|^{p-2} \nabla w\cdot \nabla \varphi\ dx\ +\
\frac1p\ \int_\Omega A_t(x,w)\varphi |\nabla w|^p dx.
\end{eqnarray*}
Now, let $(u_n)_n \subset X$, $u \in X$ be such that (\ref{succ1})
and (\ref{succ2}) hold. A direct consequence of (\ref{succ2}) and
$(H_0)$ is the existence of a constant $b > 0$,
$b$ depending only on $k$ and $|u|_\infty$, such that for all $n \in \N$ and a.e. $x \in \Omega$
we have
\begin{equation}\label{bounded}
|A(x,u)| \le b, \quad |A(x,u_n)| \le  b,\quad
|A_t(x,u)| \le  b,\quad |A_t(x,u_n)| \le  b.
\end{equation}
On the other hand, by (\ref{succ1}) it follows that
\[
(u_n,\nabla u_n) \to (u,\nabla u)\qquad \hbox{in measure on $\Omega$.}
\]
Thus, being $A$ and $A_t$ Carath\'eodory functions, there results
\[
\begin{split}
&A(x,u_n)|\nabla u_n|^p\ \to\ A(x,u)|\nabla u|^p,\\
&A_t(x,u_n)|\nabla u_n|^p\ \to\ A_t(x,u)|\nabla u|^p,\\
&A(x,u_n)|\nabla u_n|^{p-2}\nabla u_n\ \to\ A(x,u)|\nabla u|^{p-2}\nabla u
\end{split}
\]
in measure on $\Omega$, too,
i.e., for all $\eps > 0$ it is
\begin{equation}\label{misura1}
\meas(\Omega_{n,\eps}) \to 0,\quad
\meas(\Omega^t_{n,\eps}) \to 0,\quad
\meas(\Omega^{p-1}_{n,\eps}) \to 0,
\end{equation}
where
\[\begin{split}
&\Omega_{n,\eps} = \{x \in \Omega: \big|A(x,u_n)|\nabla u_n|^p - A(x,u)|\nabla u|^p\big| \ge \eps\}, \\
&\Omega^t_{n,\eps} = \{x \in \Omega: \big|A_t(x,u_n)|\nabla u_n|^p - A_t(x,u)|\nabla u|^p\big| \ge \eps\}, \\
&\Omega^{p-1}_{n,\eps} = \{x \in \Omega: \big|A(x,u_n)|\nabla u_n|^{p-2}\nabla u_n - A(x,u)|\nabla u|^{p-2}\nabla u\big| \ge \eps\}.
\end{split}
\]
So, fixing $\eps > 0$, by applying Vitali--Hahn--Saks Theorem
and taking into account the absolutely continuity
of the Lebesgue integral, there exists $\delta_\eps > 0$ (eventually, $\delta_\eps \le \eps$),
such that if $E \subset \Omega$, $\meas(E) < \delta_\eps$, then
\begin{equation}\label{misura2}
\int_E|\nabla u|^p dx < \eps,\qquad \int_E|\nabla u_n|^p dx < \eps
\quad\hbox{for all $n\in\N$;}
\end{equation}
moreover, by \eqref{misura1} an integer $n_\eps$ exists such that
\begin{equation}\label{misura3}
\meas(\Omega_{n,\eps}) < \delta_\eps,\; \meas(\Omega^t_{n,\eps}) < \delta_\eps,\;
\meas(\Omega^{p-1}_{n,\eps}) < \delta_\eps \quad \hbox{for all $n\ge n_\eps$.}
\end{equation}
Then, from \eqref{bounded}, \eqref{misura2}, \eqref{misura3}
and direct computations it follows that
\[\begin{split}
|\bar\J(u_n) - \bar\J(u)| \ &\le\ \frac1p\ \int_{\Omega_{n,\eps}}(|A(x,u_n)| |\nabla u_n|^p + |A(x,u)| |\nabla u|^p) dx\\
&\ +\ \frac1p\ \int_{\Omega\setminus \Omega_{n,\eps}}\big| A(x,u_n)|\nabla u_n|^p - A(x,u)|\nabla u|^p\big| dx
\ <\ b_1 \eps
\end{split}
\]
for all $n\ge n_\eps$, where $b_1 > 0$ is a suitable constant independent of $\eps$.
Whence, $\bar\J(u_n) \to \bar\J(u)$.\\
Now, fixing any $\eps > 0$ and taking any $\varphi \in X$, we have
\[\begin{split}
&|\langle d\bar\J(u_n) - d\bar\J(u),\varphi\rangle|\
\le\ \int_{\Omega^{p-1}_{n,\eps}} |A(x,u_n)| |\nabla u_n|^{p-1} |\nabla \varphi|\ dx\\
&\quad + \int_{\Omega^{p-1}_{n,\eps}}|A(x,u)||\nabla u|^{p-1} |\nabla \varphi|\ dx \\
&\quad  + \int_{\Omega \setminus \Omega^{p-1}_{n,\eps}} \big|A(x,u_n)|\nabla u_n|^{p-2} \nabla u_n
- A(x,u) |\nabla u|^{p-2} \nabla u\big| |\nabla \varphi| dx \\
&\quad +\ \frac1p\ \int_{\Omega^{t}_{n,\eps}} (|A_t(x,u_n)| |\nabla u_n|^{p} + |A_t(x,u)||\nabla u|^{p}) |\varphi|\ dx \\
&\quad +\ \frac1p\ \int_{\Omega \setminus \Omega^{t}_{n,\eps}} \big|A_t(x,u_n)|\nabla u_n|^{p}
- A_t(x,u) |\nabla u|^{p}\big| |\varphi| dx.
\end{split}
\]
Thus, reasoning as above, from \eqref{space}, \eqref{bounded}, \eqref{misura2}, \eqref{misura3}
and direct computations, a constant $b_2 > 0$, $b_2$ independent of $\eps$ and $\varphi$, exists such that
\[
\begin{split}
&|\langle d\bar\J(u_n) - d\bar\J(u),\varphi\rangle|\ \le\
\big(2 b \eps^{1 - \frac1p}+ \eps (\meas(\Omega))^{1 - \frac1p}\big) \|\varphi\|\\
&\qquad + \frac{\eps}{p}\big(2 b + \meas(\Omega)\big) |\varphi|_\infty
\ \le\ b_2 \max\{\eps,\eps^{1 - \frac1p}\} \|\varphi\|_X
\end{split}
\]
for all $n$ large enough.
Hence, by the arbitrariness of $\eps$ and $\varphi \in X$, we have
\[
\| d\bar\J(u_n) - d\bar\J(u)\|_{X'}\ \to\ 0.
\]
On the other hand, from \eqref{sug3} and standard arguments
(see, e.g., \cite[Subsection 2.1]{DJM}), it follows that
the functional
\[
u \in W^{1,p}_0(\Omega) \ \mapsto\ \frac{\lambda}{p}\ \int_\Omega |u|^p dx + \int_\Omega G^\infty(x,u) dx \in \R
\]
is $C^1$ in $(W^{1,p}_0(\Omega),\|\cdot\|)$, and so in $(X,\|\cdot\|_X)$; hence, the
tesis follows.
\end{proof}

Thus, if conditions $(H_0)$, $(h_0)$ and $(h_1)$ hold, for each $p \ge 1$, problem $(P)$ has a variational
structure and its bounded weak solutions are critical points of
$\J = \J_{\lambda^\infty}$ in the Banach space $X$.

As our aim is applying variational methods to the study of critical points of $\J$ in the asymptotically $p$-linear case,
we introduce the following further conditions:
\begin{description}{}{}
\item[$(H_3)$] we have
\[
\lim_{|t| \to + \infty} A_t(x,t)\ t\ =\ 0 \quad \hbox{uniformly a.e. in $\Omega$;}
\]
\item[$(H_4)$] there exists $\alpha_1 > 0$ (without loss of generality, $\alpha_1 \le 1$) such that
\[
A(x,t) + \frac1p\, A_t(x,t)\, t \ge \alpha_1\, A(x,t) \quad \hbox{a.e. in $\Omega$, for all $t \in \R$.}
\]
\end{description}

\begin{remark}
Hypothesis $(H_2)$ implies that
\[
\liminf_{|t| \to + \infty} A_t(x,t)\ t\ =\ 0 \quad \hbox{a.e. in $\Omega$;}
\]
hence, condition $(H_3)$ is quite natural.
\end{remark}

\begin{remark}\label{rem1}
By $(H_2)$ and $(H_3)$, for each $\eps > 0$, a radius $R_{\eps} > 0$ exists such that
\begin{eqnarray}
&&|A(x,t) - A^\infty(x)|\ <\ \eps \quad \hbox{for a.e. $x\in \Omega$, if $|t| \ge R_{\eps}$,} \label{rm1}\\[5pt]
&&|A_t(x,t) t|\ <\ \eps \quad \hbox{for a.e. $x\in \Omega$, if $|t| \ge R_{\eps}$.} \label{rm2}
\end{eqnarray}
Since \eqref{rm1} implies
\[
|A(x,t)|\ \le\ |A^\infty|_\infty + \eps \quad \hbox{for a.e. $x\in \Omega$, if $|t| \ge R_{\eps}$,}
\]
it follows from $(H_0)$ and \eqref{rm2} that
\begin{equation}\label{rm3}
|A(x,t)| \le b, \quad |A_t(x,t)| \le b\qquad \hbox{for a.e. $x \in \Omega$, for all $t \in \R$,}
\end{equation}
for a suitable $b >0$.
\end{remark}

\begin{remark}\label{rem}
In the proof of Proposition \ref{smooth}, assumption \eqref{succ2} is required only for the boundedness conditions \eqref{bounded},
which are necessary for investigating the smoothness of $\bar\J$.
However, this uniform bound can be avoided in the hypotheses $(H_2)$ and $(H_3)$ as \eqref{rm3} holds.
Whence, in this set of hypotheses,
for any $p\ge 1$, the functional $\bar\J$ is continuous in $W^{1,p}_0(\Omega)$, and so is $\J_\lambda$
for any $\lambda \in \R$.
However, in general, $\J_\lambda$ is not $C^1$ in $W^{1,p}_0(\Omega)$ as
it is G\^ateaux differentiable in $u \in W^{1,p}_0(\Omega)$ only along bounded directions.
\end{remark}

Here and in the following, by $\sigma(A_p^\infty)$ we denote the spectrum of the operator
\begin{equation}\label{operator}
A_p^\infty : u \in W^{1,p}_0(\Omega)\ \mapsto\ - \divg (A^\infty(x)\ |\nabla u|^{p-2}\ \nabla u) \in W^{-1,p'}(\Omega),
\end{equation}
which is the set of $\lambda \in \R$ such that the nonlinear eigenvalue problem
\[
\left\{
\begin{array}{ll}
- \divg (A^\infty(x)\ |\nabla u|^{p-2}\ \nabla u)\ =\ \lambda\ |u|^{p-2}\ u & \hbox{in $\Omega$,}\\
u\ =\ 0 & \hbox{on $\partial\Omega$}
\end{array}
\right.
\]
has a nontrivial (weak) solution in $W^{1,p}_0(\Omega)$, i.e. some $u \in W^{1,p}_0(\Omega)$, $u \not\equiv 0$,
exists such that
\[
\int_\Omega A^\infty(x) |\nabla u|^{p-2} \nabla u \cdot \nabla\varphi\ dx \ =\ \lambda \int_\Omega |u|^{p-2}u\varphi\ dx
\quad \hbox{for all $\varphi \in W^{1,p}_0(\Omega)$.}
\]

\begin{proposition} \label{PSbound}
If $p > 1$, the hypotheses $(H_0)$--$(H_4)$, $(h_0)$ and $(h_1)$ hold, and
$\lambda \not\in \sigma(A_p^\infty)$, then for all $\beta \in \R$, each
$(PS)_\beta$--sequence of $\J_\lambda$ in $X$ is bounded in the $W^{1,p}_0$--norm.
\end{proposition}

\begin{proof}
Taking $\beta > 0$, let $(u_n)_n \subset X$ be a $(PS)_\beta$--sequence,
i.e.
\begin{equation}\label{c1}
\J_\lambda(u_n) \to \beta \quad \hbox{and}\quad \|d\J_\lambda(u_n)\|_{X'} \to 0\qquad
\mbox{if $\ n\to+\infty$.}
\end{equation}
Arguing by contradiction, we suppose that
\begin{equation}\label{infinity}
\|u_n\| \ \to\ +\infty.
\end{equation}
Hence, without loss of generality, for any $n \in \N$ we assume $\|u_n\| \ > 0$,
and define
\begin{equation}\label{infinity1}
v_n \ =\ \frac{u_n}{\|u_n\|} ,\quad \hbox{hence}\quad \|v_n\| = 1.
\end{equation}
Then, there exists $v \in W^{1,p}_0(\Omega)$ such that, up to subsequences, we have
\begin{eqnarray}
&&v_n \rightharpoonup v\ \hbox{weakly in $W^{1,p}_0(\Omega)$,}
\label{c22}\\
&&v_n \to v\ \hbox{strongly in $L^q(\Omega)$ for each $1 \le q < p^*$,}
\label{c23}\\
&&v_n \to v\ \hbox{a.e. in $\Omega$.}
\label{c24}
\end{eqnarray}
In order to yield a contradiction, we organize the proof in some steps:
\begin{itemize}
\item[1.]  $v \not\equiv 0$;
\item[2.] a constant $b_0 > 0$ exists such that for any $\mu > 0$ there exists $n_\mu \in \N$ such that
\begin{equation}\label{ps16}
\int_{\Omega\setminus \Omega_n^\mu} |\nabla v_n|^p dx\ \le\ b_0\ \max\{\mu, \mu^p\}\qquad \hbox{for all $n\ge n_\mu$,}
\end{equation}
where
\begin{equation}\label{defmu}
\Omega_n^\mu\ =\ \{x \in \Omega :\ |v_n(x)| \ge \mu\};
\end{equation}
\item[3.] taking $\Omega_0\ =\ \{x \in \Omega :\ v(x)= 0\}$,
if $\meas(\Omega_0) > 0$ then
\begin{equation}\label{aboutzero}
\int_{\Omega_0} |\nabla v|^p dx\ =\ 0,
\end{equation}
which implies $\nabla v = 0$ a.e. in $\Omega_0$ and, clearly,
\[
\int_{\Omega_0} A^\infty(x)\ |\nabla v|^{p-2}\ \nabla v\cdot \nabla\varphi\ dx\ =\
\int_{\Omega_0}\lambda\ |v|^{p-2}\ v \varphi\ dx
\quad \hbox{for all $\varphi \in W^{1,p}_0(\Omega)$;}
\]
\item[4.] taking any $\varphi \in X$ we have
\begin{equation}\label{ps21ter}
\int_{\Omega} A^\infty(x) |\nabla v_n|^{p-2} \nabla v_n \cdot \nabla \varphi\ dx
- \lambda \int_{\Omega} |v_n|^{p-2}v_n \varphi\ dx\
 \to\ 0;
 \end{equation}
\item[5.] $\lambda \in \sigma(A_p^\infty)$, in contradiction with the hypotheses.
\end{itemize}
\smallskip

\noindent
For simplicity, here and in the following $b_i$ denotes any strictly positive constant independent of $n$.
\smallskip

\noindent
{\sl Step 1.}
Firstly, let us point out that for any $\eps > 0$ from
\eqref{sobqbis}, \eqref{sug2} and \eqref{infinity1} it follows
\[
\big|\int_\Omega \frac{G^\infty(x,u_n)}{\|u_n\|^p}\ dx\big|\
\le\ \frac{L_\eps \meas(\Omega)}{\|u_n\|^p} + \eps \gamma_p^p;
\]
hence, \eqref{infinity} implies
\[
\big|\int_\Omega \frac{G^\infty(x,u_n)}{\|u_n\|^p}\ dx\big|\
\le\ \eps (1+\gamma_p^p) \quad \hbox{for all $n \ge n_\eps$}
\]
for $n_\eps$ large enough. Thus, we have
\begin{equation}\label{tozero1}
\int_\Omega \frac{G^\infty(x,u_n)}{\|u_n\|^p}\ dx\ \to\ 0.
\end{equation}
Furthermore, \eqref{c1} and \eqref{infinity} give
\begin{equation}\label{tozero2}
\frac{\J_\lambda(u_n)}{\|u_n\|^p}\ \to\ 0.
\end{equation}
Now, arguing by contradiction, assume $v \equiv 0$. Then, from \eqref{c23} it follows
\begin{equation}\label{tozero}
\int_\Omega |v_n|^p dx\ \to\ 0,
\end{equation}
but for any $n \in \N$, condition $(H_1)$, \eqref{withlambda} and \eqref{infinity1} imply
\[\begin{split}
0\ &<\ \frac{\alpha_0}{p}\ =\ \frac{\alpha_0}{p} \|v_n\|^p\ \le\
\frac1p \int_\Omega A(x,u_n) |\nabla v_n|^p dx\\
&=\ \frac{\J_\lambda(u_n)}{\|u_n\|^p} + \frac{\lambda}{p} \int_\Omega |v_n|^p dx
 + \int_\Omega \frac{G^\infty(x,u_n)}{\|u_n\|^p}\ dx\end{split}
\]
in contradiction with \eqref{tozero1}--\eqref{tozero}.
\smallskip

\noindent
{\sl Step 2.} Taking any $\phi \in X$, we have
\begin{equation}\label{ps1}\begin{split}
&\langle d\J_\lambda(u_n),\frac{\phi}{\|u_n\|^{p-1}}\rangle\ =\
\int_\Omega A(x,u_n)\ |\nabla v_n|^{p-2} \nabla v_n \cdot \nabla \phi\ dx\\
&\quad + \frac1p \int_\Omega A_t(x,u_n)\ \|u_n\|\ |\nabla v_n|^p \phi\ dx
 - \lambda \int_\Omega |v_n|^{p-2} v_n\ \phi\ dx\\
&\quad - \int_\Omega \frac{g^\infty(x,u_n)}{\|u_n\|^{p-1}}\ \phi\ dx. \end{split}
\end{equation}
Fix any $\mu > 0$ and $\eps > 0$. From one hand,
\begin{equation}\label{ps2}
\big|\langle d\J_\lambda(u_n),\frac{\phi}{\|u_n\|^{p-1}}\rangle\big| \ \le \
\frac{\|d\J_\lambda(u_n)\|_{X'}}{\|u_n\|^{p-1}} \|\phi\|
\quad \hbox{for all $n \in \N$;}
\end{equation}
while from \eqref{sobqbis}, \eqref{sug3}, \eqref{infinity1} and the H\"older inequality it follows
\begin{equation}\label{ps3}
\big|\int_\Omega \frac{g^\infty(x,u_n)}{\|u_n\|^{p-1}} \phi\ dx\big|\ \le\
\left(\eps \gamma^{p}_{p} + \frac{L_\eps \gamma_1}{\|u_n\|^{p-1}}\right) \|\phi\|
\quad \hbox{for all $n \in \N$.}
\end{equation}
On the other hand, by \eqref{infinity1} and \eqref{defmu} we have
\begin{equation}\label{ps7}
|u_n (x)|\ > \ \mu \|u_n\|
\quad \hbox{for all $x \in\Omega_n^\mu$, $\ n \in \N$.}
\end{equation}
Then, from \eqref{infinity} an integer $n_{\mu,\eps}$, independent of $\phi$, exists such that
\eqref{c1} and \eqref{ps2} imply
\begin{equation}\label{ps4}
\big|\langle d\J_\lambda(u_n),\frac{\phi}{\|u_n\|^{p-1}}\rangle\big| \ \le \ \eps \|\phi\|
\quad \hbox{for all $n \ge n_{\mu,\eps}$,}
\end{equation}
while inequality \eqref{ps3} becomes
\begin{equation}\label{ps5}
\big|\int_\Omega \frac{g^\infty(x,u_n)}{\|u_n\|^{p-1}}\ \phi\ dx\big|\ \le\ \eps \left(\gamma_p^p + 1\right)\ \|\phi\|
\quad \hbox{for all $n \ge n_{\mu,\eps}$,}
\end{equation}
and from \eqref{rm2} and \eqref{ps7} it follows
\begin{equation}
|A_t(x,u_n(x)) u_n(x)|\ <\ \eps \quad \hbox{for a.e. $x\in \Omega_n^\mu$, if $n \ge n_{\mu,\eps}$.} \label{ps9}
\end{equation}
Hence, for all $n \ge n_{\mu,\eps}$ by \eqref{infinity1}, \eqref{defmu} and \eqref{ps9}, direct computations imply
\begin{equation}\label{ps11}
\int_{\Omega_n^\mu} \frac{\big|A_t(x,u_n)u_n\big|}{|v_n|} |\nabla v_n|^p dx\
 \le\ \frac{\eps}{\mu},
\end{equation}
and then
\begin{equation}\label{ps11bis}\begin{split}
\big|\int_{\Omega_n^\mu} A_t(x,u_n) \|u_n\|\ |\nabla v_n|^p \ \phi\ dx\big|\ &\le\
\int_{\Omega_n^\mu} \frac{\big|A_t(x,u_n)u_n\big|}{|v_n|} |\nabla v_n|^p |\phi|\ dx\\
& \le\ \frac{\eps}{\mu} |\phi|_\infty.\end{split}
\end{equation}
Now, for any $n \in \N$, let us consider the cut--off function $T_\mu : \R \to \R$
such that
\[
T_\mu(t)\ =\ \left\{\begin{array}{ll}
t&\hbox{if $|t| < \mu$,}\\
\mu \frac{t}{|t|}&\hbox{if $|t| \ge \mu$.}
\end{array}\right.
\]
As
\begin{eqnarray}\label{ps12}
T_\mu(v_n(x)) &=& \left\{\begin{array}{ll}
v_n(x)&\hbox{for a.e. $x \in \Omega \setminus \Omega_n^\mu$,}\\
\mu \frac{v_n(x)}{|v_n(x)|}&\hbox{if $x \in \Omega_n^\mu$,}
\end{array}\right.\\
\nabla T_\mu (v_n(x)) &=& \left\{\begin{array}{ll}
\nabla v_n(x)&\hbox{for a.e. $x \in \Omega \setminus \Omega_n^\mu$,}\\
0 &\hbox{for a.e. $x \in \Omega_n^\mu$,}
\end{array}\right.\label{ps13}
\end{eqnarray}
then $T_\mu(v_n) \in X$ with
\begin{equation}\label{ps14}
\|T_\mu (v_n)\| \ \le \ 1, \qquad |T_\mu (v_n)|_\infty \ \le \ \mu .
\end{equation}
Thus, applying \eqref{ps1} on the test function $\phi = T_\mu (v_n)$, we have
\[\begin{split}
&\int_\Omega A(x,u_n) |\nabla v_n|^{p-2} \nabla v_n \cdot \nabla T_\mu (v_n) dx +
\frac1p \int_\Omega A_t(x,u_n) \|u_n\| |\nabla v_n|^p T_\mu (v_n) dx\\
&\; = \langle d\J_\lambda(u_n),\frac{T_\mu (v_n)}{\|u_n\|^{p-1}}\rangle
 + \lambda \int_\Omega |v_n|^{p-2}v_n T_\mu (v_n) dx
+ \int_\Omega \frac{g^\infty(x,u_n)}{\|u_n\|^{p-1}} T_\mu (v_n) dx,
\end{split}
\]
where \eqref{infinity1}, \eqref{ps12} and \eqref{ps13} imply
\[\begin{split}
&\int_\Omega A(x,u_n) |\nabla v_n|^{p-2} \nabla v_n \cdot \nabla T_\mu (v_n) dx +
\frac1p \int_\Omega A_t(x,u_n) \|u_n\| |\nabla v_n|^p T_\mu (v_n) dx\\
&\; = \int_{\Omega\setminus \Omega_n^\mu} A(x,u_n) |\nabla v_n|^{p} dx
+ \frac1p \int_{\Omega\setminus \Omega_n^\mu} A_t(x,u_n) \|u_n\| v_n |\nabla v_n|^p dx\\
&\quad + \frac{\mu}p \int_{\Omega_n^\mu} A_t(x,u_n) \|u_n\| \frac{v_n}{|v_n|} |\nabla v_n|^p dx\\
&\; = \int_{\Omega\setminus \Omega_n^\mu} \big( A(x,u_n) + \frac1p A_t(x,u_n) u_n \big)|\nabla v_n|^p dx
+ \frac{\mu}p \int_{\Omega_n^\mu} \frac{A_t(x,u_n) u_n}{|v_n|} |\nabla v_n|^p dx,
\end{split}
\]
while \eqref{sobq} with $q = p-1$, \eqref{infinity1}, \eqref{defmu} and \eqref{ps12} give
\[
\big|\int_\Omega |v_n|^{p-2}v_n T_\mu (v_n) dx\big|\ \le\ \mu^p \meas(\Omega) + \mu \gamma_{p-1}^{p-1}.
\]
Whence, from \eqref{ps4}, \eqref{ps5}, \eqref{ps11} and \eqref{ps14} it follows
\begin{equation}\label{ps15}\begin{split}
&\int_{\Omega\setminus \Omega_n^\mu} \big( A(x,u_n) + \frac1p A_t(x,u_n) u_n \big)|\nabla v_n|^p dx\
\le\ \eps (\gamma_p^p + 2 + \frac{1}{p})\\
 &\qquad + |\lambda| (\mu^{p} \meas(\Omega) + \mu \gamma_{p-1}^{p-1})
 \qquad \hbox{for all $n\ge n_{\mu,\eps}$.}
 \end{split}
\end{equation}
As $\mu$ and $\eps$ are any and independent one from the other,
we can fix $\eps = \mu$; hence, $n_\mu = n_{\mu,\mu}$ and
\eqref{ps15} becomes
\begin{equation}\label{ps15bis}
\int_{\Omega\setminus \Omega_n^\mu} \big( A(x,u_n) + \frac1p A_t(x,u_n) u_n \big)|\nabla v_n|^p dx\
\le\ b_1\ \max\{\mu, \mu^p\}
\end{equation}
for all $n\ge n_\mu$, where $b_1 = \gamma_p^p + 2 + \frac{1}{p} + |\lambda| \meas(\Omega) + |\lambda|\gamma_{p-1}^{p-1} > 0$.\\
Vice versa, by assumptions $(H_1)$ and $(H_4)$ we have
\[\begin{split}
&\alpha_0 \alpha_1
\int_{\Omega\setminus \Omega_n^\mu} |\nabla v_n|^p dx\ \le\ \alpha_1
\int_{\Omega\setminus \Omega_n^\mu}  A(x,u_n) |\nabla v_n|^p dx\\
&\qquad \le\
\int_{\Omega\setminus \Omega_n^\mu} \big( A(x,u_n) + \frac1p A_t(x,u_n) u_n \big)|\nabla v_n|^p dx;
\end{split}
\]
whence, summing up, \eqref{ps15bis} implies \eqref{ps16} with $b_0 = \frac{b_1}{\alpha_0 \alpha_1}$.
\smallskip

\noindent
{\sl Step 3.}
Firstly, we claim that if $\meas(\Omega_0) > 0$ then for any $\mu > 0$ there exists $n^\mu \in \N$ such that
\begin{equation}\label{null}
\meas(\Omega_0 \cap \Omega_n^\mu) = 0 \quad \hbox{for all $n \ge n^\mu$.}
\end{equation}
In fact, arguing by contradiction, we assume that $\bar\mu > 0$ exists such that,
up to subsequences,
\[
\meas(\Omega_0 \cap \Omega_n^{\bar\mu}) > 0 \quad \hbox{for all $n \in \N$.}
\]
From \eqref{c24} a set $\bar\Omega \subset\Omega$ exists such that
$\meas(\bar\Omega) = 0$ and $v_n(x) \to v(x)$ for all $x \not\in \bar\Omega$;
whence, for all $n \in \N$ it results
$\meas((\Omega_0 \cap \Omega_n^{\bar\mu})\setminus \bar\Omega) > 0$
and for all $x \in (\Omega_0 \cap \Omega_n^{\bar\mu})\setminus \bar\Omega$
we have both $|v_n(x)| \ge \bar\mu$ for all $n \in \N$ and $v_n(x) \to 0$ as $n \to +\infty$:
a contradiction.\\
Now, from {\sl Step 2}, \eqref{ps16} and \eqref{null} imply that
\[
\int_{\Omega_0} |\nabla v_n|^p dx\ =\ \int_{\Omega_0\setminus \Omega_n^\mu} |\nabla v_n|^p dx\ \le\
\int_{\Omega\setminus \Omega_n^\mu} |\nabla v_n|^p dx\ \le\
b_0\ \max\{\mu, \mu^p\}
\]
for all $n$ large enough, where from the weak lower semi--continuity of norms
we have
\[
\int_{\Omega_0} |\nabla v|^p dx\ \le\
\liminf_{n\to+\infty}\int_{\Omega_0} |\nabla v_n|^p dx\ \le\ b_0\ \max\{\mu, \mu^p\}.
\]
Hence, for the arbitrariness of $\mu > 0$, \eqref{aboutzero} holds.
\smallskip

\noindent
{\sl Step 4.}
Fixing any $\rho > 0$, we introduce another cut--off function $\chi_\rho \in C^1(\R,\R)$ which has to be
even, nondecreasing in $[0,+\infty[$ and such that
\[
\chi_\rho(t)\ =\ \left\{\begin{array}{ll}
0&\hbox{if $|t| < \rho$,}\\
1 &\hbox{if $|t| \ge 2\rho$,}
\end{array}\right.
\qquad \hbox{with}\quad
|\chi_\rho'(t)| \le 2 \; \hbox{for all $t \in \R$.}
\]
Taking any $\varphi \in X$, $n \in \N$,
we denote $\omega_{\rho,n} = \chi_\rho(v_n) \varphi$, hence, by definition,
\begin{eqnarray}
\omega_{\rho,n}(x) &=& \left\{\begin{array}{ll}
0&\hbox{if $|v_n(x)| < \rho$,}\\
\varphi(x)&\hbox{if $|v_n(x)| > 2\rho$,}
\end{array}\right. \label{om1}\\
\nabla \omega_{\rho,n}(x) &=& \left\{\begin{array}{ll}
0&\hbox{if $|v_n(x)| < \rho$,}\\
\nabla\varphi(x)&\hbox{if $|v_n(x)| > 2\rho$,}
\end{array}\right. \label{om2}
\end{eqnarray}
so direct computations imply $\omega_{\rho,n} \in X$ with
\begin{equation}\label{ps17ter}
\|\omega_{\rho,n}\| \le \|\varphi\| + 2 |\varphi|_\infty \le 3 \|\varphi\|_X,\quad
|\omega_{\rho,n}|_\infty \le |\varphi|_\infty \le \|\varphi\|_X.
\end{equation}
Thus, we consider the test function $\phi = \omega_{\rho,n}$ in \eqref{ps1}, and, from \eqref{defmu} with $\mu =\rho$,
we have
\begin{equation}\label{ps19}
\begin{split}
&\int_{\Omega_n^\rho} A(x,u_n) |\nabla v_n|^{p-2} \nabla v_n \cdot \nabla \omega_{\rho,n} dx
- \lambda \int_{\Omega_n^\rho} |v_n|^{p-2}v_n \omega_{\rho,n} dx\\
&\quad =\ \langle d\J_\lambda(u_n),\frac{\omega_{\rho,n}}{\|u_n\|^{p-1}}\rangle
 - \frac1p \int_{\Omega_n^\rho} A_t(x,u_n) \|u_n\| |\nabla v_n|^p \omega_{\rho,n} dx \\
&\qquad + \int_\Omega \frac{g^\infty(x,u_n)}{\|u_n\|^{p-1}} \omega_{\rho,n} dx.
\end{split}
\end{equation}
Whence, by using \eqref{ps4} with $\phi = \omega_{\rho,n}$ and $\eps =\rho$,
\eqref{ps5} with $\phi = \omega_{\rho,n}$ and $\eps =\rho$,
\eqref{ps11bis} with $\phi = \omega_{\rho,n}$, $\eps =\rho^2$ and $\mu =\rho$,
equation \eqref{ps19} with estimates \eqref{ps17ter} implies
\begin{equation}\label{ps20}
\begin{split}
&\big|\int_{\Omega_n^\rho} A(x,u_n) |\nabla v_n|^{p-2} \nabla v_n \cdot \nabla \omega_{\rho,n} dx
- \lambda \int_{\Omega_n^\rho} |v_n|^{p-2}v_n \omega_{\rho,n} dx\big|\\
&\qquad \le\ \rho \|\omega_{\rho,n}\| +\ \frac{\rho}{p}\ |\omega_{\rho,n}|_\infty
+ \rho (\gamma^p_p + 1) \|\omega_{\rho,n}\|\ \le\ \rho\ b_2\ \|\varphi\|_X
\end{split}
\end{equation}
for all $n \ge n^1_\rho$, with $n^1_\rho$ large enough and
$b_2 = 3 \gamma^p_p + 6 + \frac1p$.\\
On the other hand, by \eqref{rm1} with any $\eps > 0$, \eqref{infinity} and \eqref{ps7} with
$\mu = \eps$ (and with $\Omega_n^\eps$ as in \eqref{defmu}), an integer $n^\eps$ exists such that
\begin{equation}
|A(x,u_n(x)) - A^\infty(x)|\ <\ \eps \quad \hbox{for a.e. $x\in \Omega_n^{\eps}$, if $n \ge n^\eps$,} \label{ps8}
\end{equation}
then, taking any $\phi \in X$, for all $n \ge n^{\eps}$ by \eqref{infinity1} and \eqref{ps8}, the H\"older inequality and direct computations imply
\begin{equation}\label{ps10bis}
\big|\int_{\Omega_n^{\eps}} \big(A(x,u_n) - A^\infty(x)\big) |\nabla v_n|^{p-2} \nabla v_n \cdot \nabla \phi\ dx\big|
\le \eps \|\phi\|.
\end{equation}
In particular, if we take $\eps = \rho$ and $\phi = \omega_{\rho,n}$ in \eqref{ps10bis}, an integer
$n^2_\rho \ge n^1_\rho$ is such that from \eqref{ps17ter} and \eqref{ps20}
it follows
\begin{equation}\label{ps10}\begin{split}
&\big|\int_{\Omega_n^{\rho}} A^\infty(x) |\nabla v_n|^{p-2} \nabla v_n \cdot \nabla \omega_{\rho,n} dx
\ -\ \lambda \int_{\Omega^{\rho}_n} |v_n|^{p-2}v_n \omega_{\rho,n}\ dx\big|\\
&\qquad \le \rho b_3 \|\varphi\|_X\end{split}
\end{equation}
for all $n \ge n^2_{\rho}$, with $b_3 = 3 + b_2$.\\
Now, from definitions \eqref{defmu} with $\mu = 2\rho$, direct computations
and \eqref{om1}, \eqref{om2} imply
\[\begin{split}
&\big|\int_{\Omega} A^\infty(x) |\nabla v_n|^{p-2} \nabla v_n \cdot \nabla \varphi dx
- \lambda \int_{\Omega} |v_n|^{p-2}v_n \varphi\ dx\big|\\
&\quad \le\
\big|\int_{\Omega\setminus \Omega^{2\rho}_n} A^\infty(x) |\nabla v_n|^{p-2} \nabla v_n \cdot \nabla ((1-\chi_\rho(v_n))\varphi)\ dx\big| \\
&\qquad + |\lambda| \big| \int_{\Omega\setminus \Omega^{2\rho}_n} |v_n|^{p-2}v_n (1-\chi_\rho(v_n))\varphi\ dx\big|\\
&\qquad + \big|\int_{\Omega^{\rho}_n} A^\infty(x) |\nabla v_n|^{p-2} \nabla v_n \cdot \nabla \omega_{\rho,n}\ dx
 - \lambda \int_{\Omega^{\rho}_n} |v_n|^{p-2}v_n \omega_{\rho,n}\ dx\big|,
\end{split}
\]
where \eqref{ps16} with $\mu = 2\rho$ (in {\sl Step 2}), \eqref{ps17ter}, $(H_2)$ and the H\"older inequality give
\[\begin{split}
& \big|\int_{\Omega\setminus \Omega^{2\rho}_n} A^\infty(x) |\nabla v_n|^{p-2} \nabla v_n \cdot \nabla ((1-\chi_\rho(v_n))\varphi) dx\big|\\
&\quad \le\ |A^\infty|_\infty \left(\int_{\Omega\setminus \Omega^{2\rho}_n}  |\nabla v_n|^{p} dx\right)^{1-\frac{1}{p}} \|\varphi -\omega_{\rho,n}\|\\
&\quad \le\ b_4 \ \max\{\rho^{p-1},\rho^{1-\frac{1}{p}}\}\  \|\varphi\|_X
\end{split}
\]
for all $n \ge n^3_\rho$, with $n^3_\rho$ large enough and $b_4 > 0$ independent of both $\rho$ and $\varphi$,
while \eqref{sobqbis} implies
\[\begin{split}
&\big| \int_{\Omega\setminus \Omega^{2\rho}_n} |v_n|^{p-2}v_n (1-\chi_\rho(v_n))\varphi\ dx\big|\
\le \ (2 \rho)^{p-1}\ \int_{\Omega\setminus \Omega^{2\rho}_n} |\varphi| dx\\
& \quad \le \ \rho^{p-1}\ 2^{p-1} \gamma_1 \|\varphi\|\ \le \ \rho^{p-1}\ 2^{p-1} \gamma_1 \|\varphi\|_X.
\end{split}
\]
Whence, taking $n_\rho \in \N$ large enough, from \eqref{ps10} it follows
\begin{equation}\label{ps21}
\begin{split}
&\big|\int_{\Omega} A^\infty(x) |\nabla v_n|^{p-2} \nabla v_n \cdot \nabla \varphi dx
- \lambda \int_{\Omega} |v_n|^{p-2}v_n \varphi dx\big|\\
&\qquad \le\ \max\{\rho^{1-\frac{1}{p}}, \rho, \rho^{p-1}\}\ b_5\ \|\varphi\|_X \qquad \hbox{for all $n \ge n_\rho$,}
\end{split}
\end{equation}
with $b_5 > 0$ independent of both $\rho$ and $\varphi$.
Thus, from the arbitrariness of $\rho$, \eqref{ps21} implies \eqref{ps21ter}.
\smallskip

\noindent
{\sl Step 5.}
Firstly, we apply \eqref{ps21ter} to $\varphi = v_n - v$ by taking into account \eqref{c23},
then by considering \eqref{c22} we have
\[
\int_{\Omega} A^\infty(x) \big(|\nabla v_n|^{p-2} \nabla v_n - |\nabla v|^{p-2} \nabla v\big)\cdot
\big(\nabla v_n - \nabla v\big)\ dx\ \to\ 0.
\]
Whence, from the properties of $A^\infty$ and the uniform convexity of
$(W^{1,p}_0(\Omega), \|\cdot\|)$ (as $p > 1$) it follows that
$\|v_n - v\| \to 0$. Thus,
\[
\int_{\Omega} A^\infty(x) |\nabla v_n|^{p-2} \nabla v_n \cdot \nabla \varphi dx \ \to\
\int_{\Omega} A^\infty(x) |\nabla v|^{p-2} \nabla v \cdot \nabla \varphi dx
\]
for any $\varphi \in W^{1,p}_0(\Omega)$, and by
\eqref{c23} and \eqref{ps21ter} it results
\[
\int_{\Omega} A^\infty(x) |\nabla v|^{p-2} \nabla v \cdot \nabla \varphi dx
\ =\  \lambda \int_{\Omega} |v|^{p-2}v \varphi dx,
\]
for any $\varphi \in X$, or better any $\varphi \in W^{1,p}_0(\Omega)$.
\end{proof}

As pointed out in Remark \ref{rem}, even if $A$ and $A_t$ are bounded,
we cannot simply replace $X$ with $W^{1,p}_0(\Omega)$, so the classical
Palais--Smale condition for $\J$ in $X$ requires the convergence not only
in the $W^{1,p}_0$--norm, but also in the $L^\infty$--norm. This problem
can be overcome if $p > N$ since then $X = W^{1,p}_0(\Omega)$ and the two norms $\|\cdot\|$ and $\|\cdot\|_X$ are equivalent.

\begin{proposition} \label{PScondition}
If $p > N$ and the hypotheses $(H_0)$--$(H_4)$, $(h_0)$ and $(h_1)$ hold, then for any $\lambda \not\in \sigma(A_p^\infty)$, the functional $\J_\lambda$ satisfies the $(PS)_\beta$ condition in $W^{1,p}_0(\Omega)$ at each level $\beta \in \R$.
\end{proposition}

\begin{proof}
Taking $\beta > 0$, let $(u_n)_n \subset W^{1,p}_0(\Omega)$ be a $(PS)_\beta$--sequence, i.e. \eqref{c1} holds.
From Proposition \ref{PSbound} and \eqref{sobqter} a constant $L > 0$ exists such that
\begin{equation}\label{ps25}
\|u_n\| \ \le\ L\quad\hbox{and}\quad |u_n|_\infty \ \le\ \gamma_\infty L  \qquad \hbox{for all $n \in \N$.}
\end{equation}
Hence, up to subsequences, there exists $u \in W^{1,p}_0(\Omega)$ such that
\begin{eqnarray}
&&u_n \rightharpoonup u\ \hbox{weakly in $W^{1,p}_0(\Omega)$,}
\label{c2}\\
&&u_n \to u\ \hbox{strongly in $L^q(\Omega)$ for each $q \ge 1$,}
\label{c3}\\
&&u_n \to u\ \hbox{a.e. in $\Omega$,}
\label{c4}
\end{eqnarray}
and $h \in L^p(\Omega)$ exists such that
\begin{equation}\label{infinity2}
|u_n(x)| \ \le\ h(x)\quad \hbox{a.e. in $\Omega$, for all $n \in \N$.}
\end{equation}
We claim that $u_n \to u$ strongly in $W^{1,p}_0(\Omega)$.\\
This proof is essentially as in {\sl Step 4.} of the proof
of \cite[Proposition 4.6]{CP} and follows some arguments in \cite{AB1}
according to an idea introduced in \cite{BMP}. Anyway, for completeness,
here we prove it.\\
Let us consider the real map $\psi(t) = t {\rm e}^{\eta t^2}$,
where $\eta > (\frac{\beta_2}{2\beta_1})^2$ will be fixed once
$\beta_1$, $\beta_2 > 0$ are chosen, later on, in a suitable way.
By definition,
\begin{equation}\label{eq4}
\beta_1 \psi'(t) - \beta_2 |\psi(t)| > \frac{\beta_1} 2\qquad \hbox{for all $t \in \R$.}
\end{equation}
Taking $w_{n}= u_n - u$, from \eqref{ps25} it follows
\[
|w_n|_\infty\ \le\ \gamma_\infty L + |u|_\infty;
\]
moreover, \eqref{c2} -- \eqref{infinity2} imply
\begin{eqnarray}
\label{cc2}
&&w_{n} \rightharpoonup 0\quad \hbox{weakly in $W^{1,p}_0(\Omega)$,}\\
&&w_{n} \to 0\quad \hbox{strongly in $L^q(\Omega)$ for all $q \ge 1$,}\nonumber\\
&&w_{n} \to 0 \quad \hbox{ a.e. in $\Omega$,}\nonumber
\end{eqnarray}
and $|w_n(x)| \le h(x) + |u(x)|$ a.e. in $\Omega$, for all $n \in \N$, with $h + |u| \in L^p(\Omega)$.
Hence, $\sigma_0 > 0$ exists such that
\begin{eqnarray}
&&|\psi(w_{n})| \le \sigma_0,\quad 0<\psi'(w_{n}) \le \sigma_0 \qquad\hbox{a.e. in $\Omega$, for all $n\in\N$,}\label{stim1}\\
&&\psi(w_{n}) \to 0, \quad
\psi'(w_{n}) \to 1 \qquad\hbox{a.e. in $\Omega$ if $n\to +\infty$.}\label{stim2}
\end{eqnarray}
Thus, $(\psi(w_{n}))_n$ is bounded in $W^{1,p}_0(\Omega)$, and \eqref{c1} implies
\begin{equation}\label{ip3}
\langle d\J_\lambda(u_n),\psi(w_{n})\rangle \to 0 \quad \hbox{as $n\to +\infty$,}
\end{equation}
where it is
\begin{equation}\label{ip1}
\begin{split}
&\langle d\J_\lambda(u_n),\psi(w_{n})\rangle\ =\
 \int_{\Omega} \psi'(w_{n})\ A(x,u_n) |\nabla u_n|^{p-2} \nabla u_n \cdot \nabla w_{n}\ dx\\
& \quad +\ \frac1p\ \int_{\Omega} A_t(x,u_n) \psi(w_{n}) |\nabla u_n|^p dx
- \lambda \int_{\Omega} |u_n|^{p-2} u_n \psi(w_{n})\ dx\\
&\quad -\ \int_{\Omega} g^\infty(x,u_n)\psi(v_{k,n})\ dx.
\end{split}
\end{equation}
By \eqref{sug3}, \eqref{c4}, \eqref{infinity2}, \eqref{stim1} and \eqref{stim2},
the Lebesgue Dominated Convergence Theorem implies
\[
\int_{\Omega} |u_n|^{p-2} u_n \psi(w_{n}) dx\ \to\ 0,\qquad
\int_{\Omega} g^\infty(x,u_n)\psi(w_{n}) dx\ \to\ 0;
\]
whence, by \eqref{ip3} and \eqref{ip1} we have
\begin{equation}\label{ip2}
\begin{split}
&\int_{\Omega} \psi'(w_{n})\ A(x,u_n) |\nabla u_n|^{p-2} \nabla u_n \cdot \nabla w_{n}\ dx\\
&\qquad +\ \frac1p\ \int_{\Omega} A_t(x,u_n) \psi(w_{n}) |\nabla u_n|^p dx\ =\  \eps_{1,n},
\end{split}
\end{equation}
with $\eps_{1,n} \to 0$.
On the other hand, from $(H_1)$ and \eqref{rm3}
it follows
\[
\begin{split}
&\big|\int_{\Omega} A_t(x,u_n) \psi(w_{n}) |\nabla u_n|^p dx\big|\ \le\ \frac{b}{\alpha_0}
\int_{\Omega} A(x,u_n) |\psi(w_{n})|\ |\nabla u_n|^p dx\\
&\quad =\ \frac{b}{\alpha_0}
\int_{\Omega} A(x,u_n) |\psi(w_{n})|\ |\nabla u_n|^{p-2} \nabla u_n \cdot \nabla w_n\ dx\\
&\qquad + \frac{b}{\alpha_0}
\int_{\Omega} A(x,u_n) |\psi(w_{n})|\ |\nabla u_n|^{p-2} \nabla u_n \cdot \nabla u\ dx,
\end{split}
\]
where \eqref{rm3}, H\"older inequality, \eqref{ps25}, \eqref{stim1}, \eqref{stim2},
and the Lebesgue Dominated Convergence Theorem give
\[
\int_{\Omega} A(x,u_n) |\psi(w_{n})|\ |\nabla u_n|^{p-2} \nabla u_n \cdot \nabla u\  dx\ \to \ 0.
\]
Whence, a sequence $\eps_{2,n} \to 0$ exists such that from \eqref{ip2} and the above estimates it follows
\begin{equation}\label{stim3}
\eps_{2,n}\ \ge\ \int_{\Omega} \big(\psi'(w_{n}) - \frac{b}{p \alpha_0} |\psi(w_{n})|\big)
A(x,u_n) |\nabla u_n|^{p-2} \nabla u_n \cdot \nabla w_{n}\ dx
\end{equation}
for all $n \in \N$.
Now, taking $\beta_1 =1$ and $\beta_2 = \frac{b}{p \alpha_0}$ in the definition of $\psi$,
and denoting $h_{n} = \beta_1 \psi'(w_{n}) - \beta_2|\psi(w_{n})|$,
from \eqref{eq4} and \eqref{stim1} it follows
\begin{equation}\label{stim10}
\frac12 \le h_{n}(x) \le \sigma_0(1 + \beta_2)
\quad\hbox{a.e. in $\Omega$, for all $n \in \N$;}
\end{equation}
while from (\ref{stim2}) it is
\begin{equation}\label{stim4}
h_{n}(x) \to 1\quad\hbox{a.e. in $\Omega$, as $n \to +\infty$.}
\end{equation}
Moreover, it is
\begin{eqnarray*}
&&\int_{\Omega}
h_{n}\ A(x,u_n) |\nabla u_n|^{p-2} \nabla u_n \cdot \nabla w_{n} dx
 =
\int_{\Omega} A(x,u) |\nabla u|^{p-2} \nabla u \cdot \nabla w_{n} dx\\
&&\qquad +
\int_{\Omega}
\big(h_{n}\ A(x,u_n) - A(x,u)\big)\ |\nabla u|^{p-2} \nabla u \cdot \nabla w_{n}\ dx\\
&&\qquad +
\int_{\Omega}
h_{n}\ A(x,u_n) \big(|\nabla u_n|^{p-2} \nabla u_n - |\nabla u|^{p-2} \nabla u\big) \cdot \nabla w_{n}\ dx,
\end{eqnarray*}
where \eqref{cc2} implies
\[
\int_{\Omega} A(x,u) |\nabla u|^{p-2} \nabla u \cdot \nabla w_{n}\ dx\ \to\ 0,
\]
while H\"older inequality, \eqref{ps25}, and also \eqref{rm3}, \eqref{c4},
\eqref{stim10}, \eqref{stim4} and the Lebesgue Dominated Convergence Theorem, imply
\[
\int_{\Omega}
\big(h_{n}\ A(x,u_n) - A(x,u)\big)\ |\nabla u|^{p-2} \nabla u \cdot \nabla w_{n}\ dx
	 \ \to\ 0.
\]
Thus, the convexity condition
\[
\big(|\nabla u_n|^{p-2} \nabla u_n - |\nabla u|^{p-2} \nabla u\big) \cdot \nabla w_{n} \ \ge\ 0\quad \hbox{a.e. in $\Omega$,}
\]
$(H_1)$, \eqref{stim3} and \eqref{stim10} give
\begin{eqnarray*}
\eps_{3,n} &\ge&\int_{\Omega}
h_{n}\ A(x,u_n) \big(|\nabla u_n|^{p-2} \nabla u_n - |\nabla u|^{p-2} \nabla u\big) \cdot \nabla w_{n}\ dx\\
&\ge&\displaystyle \frac{\alpha_0}2 \int_{\Omega}
\ \big(|\nabla u_n|^{p-2} \nabla u_n - |\nabla u|^{p-2} \nabla u\big) \cdot \nabla w_{n}\ dx \ge 0.
\end{eqnarray*}
for a suitable $\eps_{3,n} \to 0$.
Whence, $\|u_n - u\| \to 0$.
\end{proof}

\section{Main result}

In addition to the hypotheses $(H_0)$--$(H_4)$, $(h_0)$ and $(h_1)$, we assume
\begin{description}
\item[$(h_2)$] there exist $\lambda^0 \in \R$ and a (Carath\'eodory) function
$g^0 : \Omega \times \R \to \R$ such that
\[
f(x,t)\ =\ \lambda^0\, |t|^{p-2}\, t + g^0(x,t)
\]
and
\[
\lim_{t \to 0}\, \frac{g^0(x,t)}{|t|^{p-1}}\ =\ 0 \quad \text{uniformly a.e. in } \Omega.
\]
\end{description}
From $(h_2)$ it follows
\begin{equation}\label{dah2}
\lim_{t \to 0}\, \frac{G^0(x,t)}{|t|^{p}}\ =\ 0 \quad \text{uniformly a.e. in } \Omega,
\end{equation}
where $G^0(x,t) = \int_\Omega g^0(x,s)ds$.

Moreover, if we write $A^0(x) = A(x,0)$, then $(H_0)$ implies $A^0 \in L^\infty(\Omega)$,
while from $(H_1)$ it follows $A^0(x) \ge \alpha_0 > 0$ a.e. in $\Omega$. Furthermore,
$(H_0)$ and $(H_3)$ imply \eqref{rm3}; whence,
\begin{equation}\label{h4}
\lim_{t \to 0}\, A(x,t) = A^0(x) \quad \text{uniformly a.e. in } \Omega.
\end{equation}
For simplicity, as in \eqref{operator}, we introduce the operator
\[
A^0_p\ :\ u \in W^{1,p}_0(\Omega)\ \mapsto\ - \divg (A^0(x) |\nabla u|^{p-2} \nabla u) \in W^{-1,p'}(\Omega)
\]
and denote its spectrum by $\sigma(A^0_p)$.

For $\sharp = 0, \infty$, let
\[
I^\sharp(u)\ =\ \int_\Omega A^\sharp(x)\ |\nabla u|^p dx, \quad u \in W^{1,p}_0(\Omega),
\]
and let
\[
\M^\sharp = \bgset{u \in W^{1,p}_0(\Omega)\ :\ I^\sharp(u) = 1}.
\]
Since the hypotheses imply that
\begin{equation}\label{lim3}
A^\sharp \in L^\infty(\Omega)\qquad \hbox{and}\qquad A^\sharp(x) \ge \alpha_0 > 0 \;\ \hbox{for a.e. $x \in \Omega$,}
\end{equation}
then $\M^\sharp \subset W^{1,p}_0(\Omega) \setminus \set{0}$ is a bounded symmetric
complete $C^1$-Finsler manifold radially homeomorphic to the unit sphere in $W^{1,p}_0(\Omega)$. Let
\[
\Psi(u) = \frac{1}{\dint_\Omega |u|^p dx}, \quad u \in W^{1,p}_0(\Omega) \setminus \set{0}.
\]
Then $\lambda \in \sigma(A_p^\sharp)$ if and only if $\lambda$ is a critical value of $\restr{\Psi}{\M^\sharp}$
by the Lagrange multiplier rule.

Now, let $\F^\sharp$ denote the class of compact symmetric subsets of $\M^\sharp$ and set
\[
\lambda^\sharp_k := \inf_{M \in \F^\sharp_k}\, \max_{u \in M}\, \Psi(u), \quad k \ge 1,
\]
where $\ \F^\sharp_k = \bgset{M \in \F^\sharp : i(M) \ge k}$, and $i$ is the cohomological index.
Then $\lambda^\sharp_k \in \sigma(A_p^\sharp)$ and $0 < \lambda^\sharp_k \nearrow +\infty$
(see \cite[Proposition 3.52]{PAO}). In particular,
\begin{equation}\label{stima}
\lambda^\sharp_1 \int_\Omega |u|^p dx\ \le\ \int_\Omega A^\sharp(x) |\nabla u|^p dx\qquad
\hbox{for all $u \in W^{1,p}_0(\Omega)$.}
\end{equation}

Our main result is the following.

\begin{theorem} \label{Theorem 4.1}
Assume that $p > N$, $(H_0)$--$(H_4)$ and $(h_0)$--$(h_2)$ hold, and
\begin{itemize}
\item $A(x,\cdot)$ is an even function for a.a.\! $x \in \Omega$ and $f(x,\cdot)$ is an odd function for a.a.\! $x \in \Omega$,
\item $\lambda^\infty \not\in \sigma(A_p^\infty)$.
\end{itemize}
If $m, l \in \N$, $l\ne m$, exist such that one of the two following conditions hold:
\begin{itemize}
\item[{\sl (i)}] $l > m \quad \hbox{and}\quad \lambda^0_l < \lambda^0,\quad \lambda^\infty < \lambda^\infty_{m + 1}$;
\item[{\sl (ii)}] $l < m \quad \hbox{and}\quad
\lambda^0 < \lambda^0_{l + 1},\quad \lambda^\infty_m < \lambda^\infty$;
\end{itemize}
then problem $(P)$ has at least $|l - m|$ distinct pairs of nontrivial solutions.
\end{theorem}

From here on, let $p > N$ and assume that the hypotheses of Theorem \ref{Theorem 4.1} hold.
Thus, $X = W^{1,p}_0(\Omega)$ and, from \eqref{gei0} and condition $(h_1)$,
\[
\J(u)\ =\ \frac1p \int_\Omega \left(A(x,u)\ |\nabla u|^p - \lambda^\infty\ |u|^p\right) dx\
-\ \int_\Omega G^\infty(x,u) dx, \quad u \in W^{1,p}_0(\Omega),
\]
while from $(h_2)$,
\[
\J(u)\ =\ \frac1p \int_\Omega \left(A(x,u)\ |\nabla u|^p - \lambda^0\ |u|^p\right) dx\
-\ \int_\Omega G^0(x,u) dx, \quad u \in W^{1,p}_0(\Omega).
\]
Furthermore, for $\sharp = 0, \infty$, we write
\begin{equation}\label{diff21}
J^\sharp(u)\ =\ \frac1p \int_\Omega \left(A^\sharp(x)\ |\nabla u|^p - \lambda^\sharp\ |u|^p\right) dx\
 =\ \frac1p \left(I^\sharp(u) - \lambda^\sharp\, |u|_p^p\right);
\end{equation}
whence,
\begin{equation}\label{diff2}
\J(u) - J^\sharp(u)\ =\ \frac1p\ \int_\Omega \left(A(x,u) - A^\sharp(x)\right) |\nabla u|^p dx\ -\ \int_\Omega G^\sharp(x,u) dx,
\end{equation}
for $u \in W^{1,p}_0(\Omega)$.

In order to prove our main result, we need the following lemmas.

\begin{lemma} \label{Lemma 4.2}
For any $\eps > 0$, a suitable $r_\eps > 0$ exists such that
\begin{equation} \label{4.1}
u \in W^{1,p}_0(\Omega),\ \pnorm[\infty]{u} \le r_\eps\; \implies\;
\abs{\J(u) - J^0(u)} \le \frac{\eps}{p}\, I^0(u).
\end{equation}
\end{lemma}

\begin{proof}
Fixing any $\eps > 0$, by \eqref{h4}, respectively \eqref{dah2}, there is a $r_\eps > 0$ such that
\[
|t| \le r_\eps\ \implies\ \abs{A(x,t) - A^0(x)} \le \frac{\eps \alpha_0}{2},
\quad \abs{G^0(x,t)} \le \frac{\eps \lambda^0_1}{2 p}\ |t|^p \quad \text{a.e. in } \Omega,
\]
where $\alpha_0$ is as in $(H_1)$. Then, if $\pnorm[\infty]{u} \le r_\eps$
from \eqref{diff2}, the estimates \eqref{lim3} and \eqref{stima} imply
\[
\abs{\J(u) - J^0(u)} \le \frac{\eps}{2 p}\ \alpha_0 \int_\Omega |\nabla u|^p dx +
\frac{\eps}{2 p}\ \lambda^0_1 \int_\Omega |u|^p dx
\le \frac{\eps}{p}\ I^0(u).\quad \QED
\]
\end{proof}

\begin{lemma} \label{Lemma 4.3}
Let ${\cal K}_\infty$ be a compact subset of $\M^\infty$.
Then, for any $\eps > 0$ there exists a constant $C_\eps = C({\cal K}_\infty,\eps) > 0$ such that
\begin{equation} \label{estiminfty}
\abs{\J(R u) - J^\infty(R u)}\ <\ \frac{\eps}{p}\, I^\infty(R u) + C_\eps\quad \hbox{for all $R \ge 0$, $u \in {\cal K}_\infty$.}
\end{equation}
\end{lemma}

\begin{proof}
We organize the proof in different steps:
\begin{itemize}
\item[{\sl (a)}] if ${\cal K}$ is a compact subset of $W^{1,p}_0(\Omega)$,
taking any $\eps > 0$ there exists $\rho_\eps = \rho({\cal K},\eps) > 0$ such that
\[
\int_{\Omega^u_{\rho_\eps}} |\nabla u|^p dx \ < \ \eps\quad \hbox{for all $u \in {\cal K}$,}
\quad \hbox{with $\Omega^u_{\rho_\eps} = \{x \in \Omega: |u(x)| < \rho_\eps\}$;}
\]
\item[{\sl (b)}] if ${\cal K}$ is a compact subset of $W^{1,p}_0(\Omega)$,
taking any $\eps > 0$ there exists $R_\eps^* = R^*({\cal K},\eps) > 0$ such that
\[
\abs{\int_{\Omega} (A(x,Ru) - A^\infty(x))\ |\nabla u|^p dx} \ < \ \eps\quad \hbox{for all $R \ge R_\eps^*$, $u \in {\cal K}$;}
\]
\item[{\sl (c)}] if ${\cal K}_\infty$ is a compact subset of $\M^\infty$,
taking any $\eps > 0$ a constant $C_\eps = C({\cal K}_\infty,\eps) > 0$ exists
such that the estimate \eqref{estiminfty} holds.
\end{itemize}
{\sl Step (a)} Firstly, we claim that for any $u \in W^{1,p}_0(\Omega)$ and $\eps > 0$
there exists $r_\eps > 0$ such that
\begin{equation}\label{u1}
\int_{\Omega^u_{r_\eps}} |\nabla u|^p dx \ < \ \eps.
\end{equation}
In fact, the monotonicity property of the Lebesgue
integral implies
\begin{equation}\label{u2}
\lim_{r \to 0} \int_{\Omega^u_{r}} |\nabla u|^p dx \ = \ \int_{\Omega^u_{0}} |\nabla u|^p dx,
\quad \hbox{with $\Omega^u_0 = \{x \in \Omega: u(x)= 0\}$,}
\end{equation}
where
\begin{equation}\label{u3}
\int_{\Omega^u_{0}} |\nabla u|^p dx \ =\ 0
\end{equation}
not only if $\meas(\Omega^u_0) = 0$ but also if $\meas(\Omega^u_0) > 0$
as it is known that
\[
\meas(\{x \in \Omega: u(x) = 0, \nabla u(x) \ne 0\})\ =\ 0
\]
(see, e.g., \cite[Ex. 17, pp. 292]{Eva}).
Whence, \eqref{u1} follows from \eqref{u2} and \eqref{u3}.\\
Now, arguing by contradiction, assume that for the compact ${\cal K}$
the thesis in {\sl Step (a)} does not hold; hence, there exist a constant $\bar \eps > 0$
and a sequence $(u_n)_n \subset {\cal K}$ such that
\begin{equation}\label{u4}
\int_{\Omega_n} |\nabla u_n|^p dx \ \ge \ \bar\eps\quad \hbox{for all $n \ge 1$,}
\quad \hbox{with $\Omega_n = \{x \in \Omega: |u_n(x)| < \frac1n\}$.}
\end{equation}
As ${\cal K}$ is compact, then $\bar u \in {\cal K}$ exists
such that, up to subsequences,
\begin{equation}\label{u7}
\|u_n - \bar u\| \to 0, \quad \hbox{and so}\quad u_n \to \bar u\;\ \hbox{a.e. in $\Omega$.}
\end{equation}
Now, taking $\eps < \bar\eps$, from \eqref{u1} applied to $\bar u$, there exists
$\bar r > 0$ such that
\[
\int_{\Omega^{\bar{u}}_{\bar{r}}} |\nabla \bar{u}|^p dx \ < \ \frac{\eps}{2}.
\]
Then, taking a $\rho < \bar{r}$, if $n$ is large enough,
not only we have $\Omega_n \subset \Omega^{u_n}_{\rho}$
but also from \eqref{u7} it follows that
\[
\int_{\Omega^{u_n}_{\rho}} |\nabla u_n|^p dx\ <\ \eps
\]
in contradiction with \eqref{u4}. \\
{\sl Step (b)} For the compacteness of ${\cal K}$, a constant $\gamma_{\cal K} > 0$ exists
such that
\begin{equation}\label{stima2}
\|u\|^p\ \le\ \gamma_{\cal K}\quad \hbox{for all $u \in {\cal K}$.}
\end{equation}
Furthermore, taking $\eps > 0$, let $\rho_\eps > 0$ be as in {\sl Step (a)} so that
\begin{equation}\label{u9}
\int_{\Omega^u_{\rho_\eps}} |\nabla u|^p dx \ < \ \frac{\eps}{2 (b + |A^\infty|_\infty)}\qquad \hbox{for all $u \in {\cal K}$,}
\end{equation}
where $b > 0$ is as in \eqref{rm3}.
On the other hand,  from $(H_2)$, a constant $\sigma_\eps > 0$ exists such that
\begin{equation}\label{u10}
\abs{A(x,t) - A^\infty(x)} \ < \ \frac{\eps}{2 \gamma_{\cal K}}\quad \hbox{for a.e. $x \in \Omega$, if $|t| \ge \sigma_\eps$,}
\end{equation}
then, taking $R_\eps^* = \frac{\sigma_\eps}{\rho_\eps}$,
for all $u \in {\cal K}$, $R \ge R_\eps^*$, from \eqref{rm3}, \eqref{stima2} -- \eqref{u10} we have
\[\begin{split}
&\abs{\int_{\Omega} (A(x,Ru) - A^\infty(x))\ |\nabla u|^p dx} \ \le \
\int_{\Omega^u_{\rho_\eps}} (|A(x,Ru)| + |A^\infty(x)|)\ |\nabla u|^p dx\\
&\qquad +\ \int_{\Omega\setminus\Omega^u_{\rho_\eps}} |A(x,Ru) - A^\infty(x)|\ |\nabla u|^p dx \ <\ \eps.
\end{split}
\]
{\sl Step (c)} Consider ${\cal K}_\infty$, compact subset of $\M^\infty$, and take any $\eps > 0$.\\
Firstly, let us remark that, by \eqref{sug2}, there is a $L_\eps > 0$ such that
\[
\abs{G^\infty(x,t)}\ \le\ \frac{\eps \lambda^\infty_1}{2 p}\, |t|^p + L_\eps \quad \text{for a.a. } x \in \Omega \text{ and all } t \in \R.
\]
Hence, \eqref{stima} implies that
\begin{equation}\label{onG}
\abs{\int_\Omega G^\infty(x,u)\, dx}\ \le\
 \frac{\eps}{2 p}\, I^\infty(u) + L_\eps \meas(\Omega)\quad\hbox{for all $u \in W^{1,p}_0(\Omega)$.}
\end{equation}
Now, taking $u \in {\cal K}_\infty$, from {\sl Step (b)} applied to ${\cal K}_\infty$
and $\frac{\eps}2$, a constant $R^* > 0$ exists such that two cases may occur.\\
If $R \le R^*$, then \eqref{rm3}, \eqref{lim3}, \eqref{diff2} and \eqref{onG} imply that
\begin{equation}\label{lower}
\abs{\J(R u) - J^\infty(R u)}\ \le\ \frac{(R^*)^p}{p}\ (b + |A^\infty|_\infty) \gamma_{{\cal K}_\infty}\ +\
\frac{\eps}{2p}\, I^\infty(R u) + L_\eps \meas(\Omega).
\end{equation}
On the contrary, if $R > R^*$, then by \eqref{diff2}, \eqref{onG} and {\sl Step (b)}, as
${\cal K}_\infty \subset \M^\infty$, it follows
\begin{equation}\label{upper}
\begin{split}
&\abs{\J(R u) - J^\infty(R u)}\ \le\  R^p\ \frac{\eps}{2 p}\ +\
\frac{\eps}{2p}\ I^\infty(R u) + L_\eps \meas(\Omega)\\
&\qquad =\  R^p\ \frac{\eps}{2 p}\ I^\infty(u)\ +\
\frac{\eps}{2p}\ I^\infty(R u) + L_\eps \meas(\Omega)\\
&\qquad =\ \frac{\eps}{p}\ I^\infty(R u) + L_\eps \meas(\Omega).
\end{split}
\end{equation}
Thus, \eqref{estiminfty} follows from \eqref{lower} and \eqref{upper}
if we choose $C_\eps > 0$ large enough.
\end{proof}

Now, we are ready to prove Theorem \ref{Theorem 4.1}.

\begin{proof}[Proof of Theorem \ref{Theorem 4.1}]
Firstly, we note that by Proposition \ref{PScondition}
$\J$ satisfies the $(PS)_\beta$ condition for all $\beta \in \R$.\\
Then, we split the proof in two steps.\\
{\sl (i)} Case $l > m$. $\; $
Let $\A$ denote the class of symmetric subsets of $W^{1,p}_0(\Omega) \setminus \set{0}$
and
\[
\A_k\ =\ \bgset{M \in \A : M \text{ is compact and } i(M) \ge k}.
\]
Set
\[
c_k := \inf_{M \in \A_k}\, \max_{u \in M}\, \J(u), \quad m + 1 \le k \le l.
\]
We will show that $- \infty < c_{m + 1} \le \dotsb \le c_l < 0$, so we can
apply Theorem \ref{Theorem 2.4}.\\
In order to see that $c_l < 0$, let $\eps > 0$ be so small that $(1 + \eps)(\lambda^0_l + \eps) < \lambda^0$.
Then, there is $M_0 \in \F^0_l$ such that $\Psi(u) \le \lambda^0_l + \eps$ for all $u \in M_0$.
Let $r = \frac{r_\eps}{\gamma_\infty}$ with $r_\eps$ as in Lemma \ref{Lemma 4.2} and $\gamma_\infty$ as in \eqref{sobqter}, and let
\[
\tilde{M}_0\ =\ \bgset{v\ =\ r \frac{u}{\norm{u}} :\ u \in M_0}.
\]
As the map $u \in M_0 \mapsto v \in \tilde{M}_0$ is an odd homeomorphism,
then $\tilde{M}_0$ is compact and $i(\tilde{M}_0) = i(M_0) \ge l$ by \ref{i2}, so
$\tilde{M}_0 \in \A_l$. By \eqref{diff21} and \eqref{4.1},
for any $v \in \tilde{M}_0$ we have
\begin{eqnarray*}
\J(v) &\le& J^0(v) + \frac{\eps}{p}\, I^0(v)\ =\ \frac1p \left((1 + \eps)\, I^0(v) -
\lambda^0 |v|_p^p\right)\\[5pt]
&=& \frac{r^p}{p\, \norm{u}^p} \left((1 + \eps)\, I^0(u) -
\lambda^0 |u|_p^p\right)\
\le\ \frac{r^p}{p\, \norm{u}^p} \left(1 + \eps - \frac{\lambda^0}{\lambda^0_l + \eps}\right) < 0;
\end{eqnarray*}
so $c_l < 0$.\\
For seeing that $c_{m + 1} > - \infty$, take any
$M_\infty \in \A_{m + 1}$ and let $\eps > 0$ be so small that
$(1 - \eps)\, \lambda^\infty_{m + 1} \ge \lambda^\infty$. Then, consider
\[
\tilde{M}_\infty\ =\ \bgset{u = v/[I^\infty(v)]^{1/p} :\, v \in M_\infty}\ \subset\ \M^\infty.
\]
As the map $v \in M_\infty \mapsto u \in \tilde{M}_\infty$ is an odd homeomorphism,
then $\tilde{M}_\infty$ is compact and $i(\tilde{M}_\infty) = i(M_\infty) \ge m + 1$ by \ref{i2}.
So, $\tilde{M}_\infty \in \F^\infty_{m + 1}$; hence,
\[
\max_{u \in \tilde{M}_\infty}\, \Psi(u) \ge \lambda^\infty_{m + 1}.
\]
Now, let $C_\eps$ be as in Lemma \ref{Lemma 4.3} with ${\cal K}_\infty = M_\infty$.
By \eqref{diff21} and \eqref{estiminfty}, for any $v \in M_\infty$, it results
\begin{eqnarray*}
\J(v) &\ge& J^\infty(v) - \frac{\eps}{p}\, I^\infty(v) - C_\eps\
=\ \frac1p \left((1 - \eps)\, I^\infty(v) - \lambda^\infty |v|_p^p\right) - C_\eps\\[5pt]
&=& \frac{I^\infty(v)}{p} \left(1 - \eps - \lambda^\infty |u|_p^p\right) - C_\eps,
\end{eqnarray*}
with $I^\infty(v) \ge 0$. Whence,
\[
\max_{v \in M_\infty} \J(v)\ \ge\ - C_\eps;
\]
thus $c_{m + 1} \ge - C_\eps$.\\
{\sl (ii)} Case $l < m$. $\; $
Let $\A^\ast$ denote the class of symmetric subsets of $W^{1,p}_0(\Omega)$,
$\Gamma$ the group of odd homeomorphisms $\gamma$ of $W^{1,p}_0(\Omega)$ such that
$\restr{\gamma}{\{\J \le 0\}}$ is the identity,
and $i^\ast$ the pseudo-index related to $i$, $\bdry{B^W_r(0)}$, and $\Gamma$,
where $W = W^{1,p}_0(\Omega)$. Then, let
\[
\A_k^\ast\ =\ \bgset{M \in \A^\ast :\, M \text{ is compact and } i^\ast(M) \ge k}
\]
and set
\[
c_k^\ast := \inf_{M \in \A_k^\ast}\, \max_{u \in M}\, \J(u), \quad l + 1 \le k \le m.
\]
We will show that $0 < c_{l + 1}^\ast \le \dotsb \le c_m^\ast < + \infty$ if $r > 0$ is sufficiently small,
and then we can apply Theorem \ref{Theorem 2.5}.\\
In order to see that $c_{l + 1}^\ast > 0$, fix $\eps > 0$ so small that $(1 - \eps)\, \lambda^0_{l + 1} > \lambda^0$,
define $r = \frac{r_\eps}{\gamma_\infty}$ with $r_\eps$ as in Lemma \ref{Lemma 4.2}
and $\gamma_\infty$ as in \eqref{sobqter}, take any $M^*_0 \in \A_{l + 1}^\ast$, and consider
\[
\tilde{M}^*_0\ =\ \bgset{u\ =\ \frac{v}{[I^0(v)]^{1/p}} :\ v \in M^*_0 \cap \bdry{B^W_r(0)}} \subset \M^0.
\]
The map $v \in M^*_0 \cap \bdry{B^W_r(0)} \mapsto u \in \tilde{M}^*_0$ is an odd homeomorphism;
hence, $\tilde{M}^*_0$ is compact and
\[
i(\tilde{M}^*_0) = i(M^*_0 \cap \bdry{B^W_r(0)}) \ge i^\ast(M^*_0) \ge l + 1
\]
by \ref{i2}. So $\tilde{M}^*_0 \in \F^0_{l + 1}$ and hence
\[
\max_{u \in \tilde{M}^*_0}\, \Psi(u) \ge \lambda^0_{l + 1}.
\]
By \eqref{diff21} and \eqref{4.1}, for any $v \in M^*_0 \cap \bdry{B^W_r(0)}$ we have
\[
\J(v)\ \ge\ J^0(v) - \frac{\eps}{p}\, I^0(v) = \frac1p \left((1 - \eps)\, I^0(v) -
\lambda^0 |v|_p^p\right)\ =\ \frac{I^0(v)}{p} \left(1 - \eps - \lambda^0 |u|_p^p\right).
\]
Since $I^0(v) \ge \alpha_0 \norm{v}^p$, it results
\[
\delta := \inf_{v \in \bdry{B^W_r(0)}}\, I^0(v) \ge \alpha_0\, r^p > 0.
\]
Whence, it follows that
\[
\max_{v \in M^*_0}\ \J(v)\ \ge\ \max_{v \in M^*_0 \cap \bdry{B^W_r(0)}}\ \J(v)\ \ge\
\frac{\delta}{p} \left(1 - \eps - \frac{\lambda^0}{\lambda^0_{l + 1}}\right) > 0;
\]
so $c_{l + 1}^\ast > 0$.\\
For proving that $c_m^\ast < + \infty$, let $\eps > 0$ be so small that $(1 + \eps)(\lambda^\infty_m + \eps) < \lambda^\infty$.
There is a $M^*_\infty \in \F^\infty_m$ such that $\Psi(u) \le \lambda^\infty_m + \eps$ for all $u \in M^*_\infty$.
Let $C_\eps$ be as in Lemma \ref{Lemma 4.3} with ${\cal K}_\infty = M^*_\infty$ and consider
\[
\tilde{M}^*_R\ = \ \bgset{v = Ru :\ u \in M^*_\infty}, \quad R > 0.
\]
The map $u \in M^*_\infty \mapsto v \in \tilde{M}^*_R$ is an odd homeomorphism;
hence, $\tilde{M}^*_R$ is compact and $i(\tilde{M}^*_R) = i(M^*_\infty) \ge m$ by \ref{i2}.
By \eqref{diff21} and \eqref{estiminfty}, for any $v \in \tilde{M}^*_R$ we have
\begin{eqnarray*}
\J(v) &\le& J^\infty(v) + \frac{\eps}{p}\, I^\infty(v) + C_\eps \
=\ \frac1p \left((1 + \eps)\, I^\infty(v) -
\lambda^\infty |v|_p^p\right) + C_\eps\\[5pt]
&=& \frac{R^p}{p} \left((1 + \eps)\, I^\infty(u) - \lambda^\infty |u|_p^p\right) + C_\eps\
 \le\ \frac{R^p}{p} \left(1 + \eps - \frac{\lambda^\infty}{\lambda^\infty_m + \eps}\right) + C_\eps.
\end{eqnarray*}
Fixing $R$ so large that the last term of the previous estimates is $\le 0$, consider
\[
\tilde{M}^*_\infty\ =\ \bgset{tv :\ v \in \tilde{M}^*_R,\, t \in [0,1]} \in \A^\ast.
\]
Since $\tilde{M}^*_R$ is compact, so is $\tilde{M}^*_\infty$.
Since $\J(v) \le 0$ on $\tilde{M}^*_R$, for any $\gamma \in \Gamma$
it results $\restr{\gamma}{\tilde{M}^*_R}$ is the identity. Thus, by
applying the piercing property \ref{i7} to
\[
\begin{split}
&A = \tilde{M}^*_R, \quad A_0 = B^W_r(0), \quad A_1 = \closure{W^{1,p}_0(\Omega) \setminus B^W_r(0)},\\[5pt]
&\varphi : (v,t)\in A \times [0,1] \mapsto \gamma(tv) \in A_0 \cup A_1
\end{split}
\]
($r$ as in the first part of the proof of this case), we have
\[
i(\gamma(M^*_\infty) \cap \bdry{B^W_r(0)}) = i(\varphi(A \times [0,1]) \cap A_0 \cap A_1)
\ge i(A) = i(\tilde{M}^*_R) \ge m.
\]
So $i^\ast(\tilde{M}^*_\infty) \ge m$, and hence $\tilde{M}^*_\infty \in \A_m^\ast$. Then,
\[
c_m^\ast \le \max_{u \in \tilde{M}^*_\infty}\, \J(u) < + \infty. \hfill\QED
\]
\end{proof}


\end{document}